\documentclass[11pt]{amsart}
\usepackage[hmarginratio=1:1]{geometry}              
\usepackage[parfill]{parskip}    

\let\oldtocsection=\tocsection
\let\oldtocsubsection=\tocsubsection
\let\oldtocsubsubsection=\tocsubsubsection
\renewcommand{\tocsection}[2]{\hspace{0em}\oldtocsection{#1}{#2}}
\renewcommand{\tocsubsection}[2]{\hspace{1em}\oldtocsubsection{#1}{#2}}
\renewcommand{\tocsubsubsection}[2]{\hspace{2em}\oldtocsubsubsection{#1}{#2}}

\usepackage{amssymb}
\usepackage{bbm}
\usepackage{amsthm}
\usepackage[normalem]{ulem}
\usepackage[all,cmtip,rotate]{xy}
\usepackage[linktocpage]{hyperref}
\usepackage{mathtools}
\usepackage{enumerate}
\usepackage{stmaryrd }
\usepackage{tikz}
\usepackage{tikz-cd}
\usetikzlibrary{arrows,decorations.markings, decorations.pathreplacing, calc}
\newcommand{\midarrow}{\draw[postaction={decorate}]}

\tikzset{partial ellipse/.style args={#1:#2:#3}{insert path={+ (#1:#3) arc (#1:#2:#3)} }}
\newcommand\vertex[1]{\fill #1 circle (.05)}

\usepackage{xcolor}
\definecolor{darkorange}{rgb}{1.0, 0.55, 0.0}
\definecolor{forestgreen}{rgb}{0.13, 0.55, 0.13}
 \definecolor{darkgreen}{rgb}{0, .7, 0}

\newtheorem{theorem}{Theorem}
\numberwithin{theorem}{section}

\newtheorem{lemma}[theorem]{Lemma}
\newtheorem{claim}[theorem]{Claim}
\newtheorem{proposition}[theorem]{Proposition}

\newtheorem{corollary}[theorem]{Corollary}

\theoremstyle{remark}
\newtheorem{remark}[theorem]{Remark}
\theoremstyle{definition}
\newtheorem{definition}[theorem]{Definition}

\newtheorem*{claim*}{Claim}
\newtheorem*{corollary*}{Corollary}

\newcommand{\R}{\mathbb{R}} 
\newcommand{\C}{\mathbb{C}}
\newcommand{\Z}{\mathbb{Z}}  
\newcommand{\inv}{^{-1}} 
\DeclareMathOperator{\lk}{lk}
\DeclareMathOperator{\st}{st}
\DeclareMathOperator{\dlk}{dlk} 
\DeclareMathOperator{\Sing}{split}

\DeclareMathOperator{\GL}{GL}

\DeclareMathOperator{\Aut}{Aut}
\DeclareMathOperator{\Out}{Out}

\DeclareMathOperator{\Isom}{Isom}

\newcommand{\G}{\Gamma}             
\newcommand{\AG}{A_\G}                 
\newcommand{\GW}{$\G$-} %
\newcommand{\WP}{{\mathcal P}}     
\newcommand{\WQ}{{\mathcal Q}}     

\newcommand{\bPi}{\Pi}  

\newcommand{\OG}{\mathcal{O}_{\Gamma}}            
\newcommand{\SaG}{\mathbb{S}_{\Gamma}}           
\newcommand{\Sa}{\mathbb{S}}                                 

\newcommand{\SP}{\Sa_\Gamma^\bPi}                                     
\newcommand{\F}{\mathcal{F}}                                    
\newcommand{\K}{\mathbb{K}}

\newcommand{\UX}{\widetilde X}

\newcommand{\iso}{\cong}   

\renewcommand{\sslash}{/\mkern-6mu/}

\title[Finite untwisted subgroups of $\Out(\AG)$]{Finite groups of untwisted outer automorphisms of RAAGs}

\author{Corey Bregman}
\address{Department of Mathematics\\ Tufts University\\ 177 College Ave\\ Medford, MA 02155\\ USA}
\email{Corey.Bregman@tufts.edu}

\author{Ruth Charney}
\address{Department of Mathematics\\ Brandeis University\\ 415 South St\\ Waltham, MA 02453\\ USA}
\email{charney@brandeis.edu}

\author{Karen Vogtmann}
\address{Department of Mathematics\\ University of Warwick\\ Coventry CV4 7AL\\ UK}
\email{K.Vogtmann@warwick.ac.uk}

\begin{document}

\maketitle

\begin{abstract}   For any right-angled Artin group $A_\G$, Charney--Stambaugh--Vogtmann showed that  the subgroup $U^0(\AG) \leq\Out(\AG)$  generated by Whitehead automorphisms and inversions acts properly and cocompactly on a contractible space $K_\G$.  In the present  paper we show that any finite subgroup of $U^0(A_\G)$ fixes a point of $K_\G$. This generalizes the fact that any finite subgroup of $\Out(F_n)$ fixes a point of Outer Space, and implies that there are only finitely many conjugacy classes of finite subgroups in $U^0(A_\G)$
\end{abstract}

\section{Introduction}

A right-angled Artin group (RAAG) is a finitely-generated group whose only defining relations are that some of the generators commute.  This can be encoded by forming a finite simplicial graph $\G$ with one vertex for each generator and an edge between each pair of commuting generators; the associated RAAG is then called $\AG$.  The extreme examples are the free group $F_n$ (if $\G$ has no edges) and the free abelian group $\Z^n$ (if $\G$ is a complete graph).  We are interested in studying finite subgroups of the group $\Out(\AG)$ of outer automorphisms of $\AG$.

For $\AG=\Z^n$, it  follows from the classical Jordan--Zassenhaus theorem that there are only finitely many conjugacy classes of finite subgroups in $\Out(\AG)=\GL(n,\Z)$ (see e.g.\cite{CR62}).  Since $\GL(n,\Z)$ acts on the symmetric space $\GL(n,\R)/O(n)$ preserving a CAT(0) metric, any finite subgroup fixes a point. Since $\GL(n,\R)/O(n)$ can be identified with the space of marked lattices $\Lambda\subset \R^n$, where a {\em marking} is a choice of basis $B,$ which gives an isomorphism $\Lambda\iso \Z^n$, it follows that any finite subgroup $G< \GL(n,\Z)$ acts by isometries on a lattice $\Lambda$.
Equivalently, any finite subgroup   $G< \GL(n,\Z)$ can be embedded in the isometry group of a flat torus $T$, so that the induced action on $\pi_1$ agrees with $G$.

For $\AG=F_n$, there is a Realization Theorem that says any finite subgroup $G$ of $\Out(F_n)$ can be realized as automorphisms of a finite graph $X$ \cite{Cul, Khr, Zim}. This means one can {\em mark} the graph by an isomorphism $\pi_1(X)\iso F_n$ so that automorphisms of $X$ induce the elements of $G$ on $\pi_1$.   Furthermore, one may assume that all vertices of $X$ have valence at least three.  Since there are only finitely many such graphs this implies that there are only finitely many conjugacy classes of finite subgroups of $\Out(F_n)$.  Thus one can study finite subgroups of $\Out(F_n)$ by studying symmetries of such graphs (see, e.g. \cite{Lev, SV}).
An equivalent way to state the Realization Theorem is that the action of the finite group $G\leq \Out(F_n)$ on Outer Space $CV_n$ has a fixed point.

In previous work  we constructed an {\em outer space} $\OG$ for an arbitrary $RAAG$ $\AG$ that combines features of  both $CV_n$ and symmetric spaces \cite{BrCV}.  The group   $\Out(\AG)$  acts on $\OG$ with finite stabilizers, and it is proved in \cite{BrCV}  that  $\OG$ is contractible.
The group $\Out(\AG)$ contains a natural {\em untwisted} subgroup $U^0(\AG)$  which is the whole group in some cases, including the case  $\AG=F_n$.  The results in  \cite{BrCV} build on the
 the fact that $\OG$ contains a  subspace $K_\Gamma$ on which the subgroup $U^0(\AG)$  acts with compact quotient.  The   space $K_\G$ was first  defined in \cite{CSV}, where it was  proved to be contractible.   Points in $K_\G$ are special types of cube complexes called  {\em $\G$-complexes} with special types of markings called {\em untwisted markings}. In this paper we prove the following theorem.
 
 {
\renewcommand{\thetheorem}{\ref{thm:NRealization}}
\begin{theorem}
 Let  $\G$ be a simplicial graph, $G$ a  finite group $G$  and $\rho\colon G\to U^0(\AG)$ a homomorphism.  Then there is a  $\Gamma$-complex $X$ with an untwisted marking $h\colon X\to\Sa_\G$ on which  $\rho$ is realized by isometries.
\end{theorem}
\addtocounter{theorem}{-1}
}
The following corollary is immediate.
 {
\renewcommand{\thetheorem}{\ref{cor:FixedPoint}}
\begin{corollary}
 Any finite subgroup of $U^0(\AG)$ has a fixed point in $K_\G$ (and therefore in $\OG$). 
\end{corollary}
\addtocounter{theorem}{-1}
}
 
 {We conjecture that the entire fixed point set is contractible, i.e. that $K_\G$ is an $\underbar{EG}$ for $G=U^0(\AG)$.  Corollary~\ref{cor:FixedPoint} is a necessary first step towards this goal.  }

 It is easy to see that there are only a finite number of combinatorial types of $\G$-complexes, generalizing the fact that there are only a finite number of combinatorial types of graphs in $CV_n$.  This gives us  the following information about finite subgroups of $U^0(\AG)$. 
 
 {
\renewcommand{\thetheorem}{\ref{cor:FinitelyMany}}
\begin{corollary}
 The group $U^0(A_\G)$ contains only finitely many conjugacy classes of finite subgroups.
\end{corollary}
\addtocounter{theorem}{-1}
}

 {Extending these theorems to all of $\Out(A_\G)$ presents subtle difficulties that we do not address in this paper.  Among these is the problem of including outer automorphisms of $A_\G$ that are induced by graph automorphisms of $\G$.  More serious is the fact that the full group $\Out(A_\G)$ may contain finite subgroups of of $\GL(n,\Z)$ which do not preserve any $\G$-complex structure, so that understanding these will require additional techniques involving the action of $\GL(n,\Z)$ on the symmetric space $\GL(n,\R)/O(n)$.}

To prove Theorem \ref{thm:NRealization}, we use an inductive approach which starts from the Realization Theorem for $\Out(F_n)$. This was inspired by work of Hensel and Kielak  \cite{HK}, who proved that a finite subgroup $G$ of $U^0(\AG)$ can be realized on some cube complex, but it is not clear whether this can be taken to be a $\G$-complex.  We borrow a number of ideas from \cite{HK}, but our proof is shorter.  In particular, much  of our proof is independent of the specific group $G$ being considered, depending rather on the combinatorial structure of the defining graph $\G$.

\subsection{Structure of the paper and outline of the proof}  In Sections 2 and 3 of the paper  we review  the   group  $U^0(\AG)$ and  the definition and basic properties of $\G$-complexes.

The strategy of the proof is to build a marked $\G$-complex realizing a finite $G<U^0(\AG)$ by gluing together marked $\Delta$-complexes for certain subgraphs $\Delta\subseteq \G$.  The subgraphs we use are those whose associated special subgroup $A_\Delta$ is invariant (up to conjugacy) under $U^0(\AG)$, which we will call {\em $U^0$-invariant subgraphs}. The argument is inductive, and the induction parameter is the {\em chain length} of $\G$, \emph{i.e.} the longest length of a chain of $U^0$-invariant subgraphs contained in $\G$.   In Section~\ref{sec:U0invariant} we study $U^0$-invariant subgraphs $\Delta$, show there is a restriction homomorphism $r_\Delta\colon U^0(\AG)\to U^0(A_\Delta)$, and
show that minimal $U^0$-invariant subgraphs are discrete, providing a base case for our induction.

   In Section~\ref{sec:XDelta} we  show that a  marked $\G$-complex that realizes a finite subgroup $G<U^0(\AG)$ contains a subcomplex associated to each $U^0$-invariant subgraph $\Delta$ with empty link, and that the restriction of the marking to this subcomplex realizes the restriction of $G$ to $U^0(A_\Delta)$.   In Section~\ref{sec:extendable} we address the opposite problem, establishing a necessary condition for extending  a $\Delta$-complex realizing the restriction of $G$   to   a $\G$-complex realizing $G$.

In Section~\ref{sec:joins} we show how to build marked $\G$-complexes when $\G$ is a simplicial join or a disjoint union of subgraphs $\Delta$ for which we already have marked $\Delta$-complexes.
  We also show that if $\G$ is  a join or disjoint union, and one can realize the restriction of a finite subgroup $G<U^0(\AG)$ on each component, then one can realize all of $G$.

Finally, in Section~\ref{sec:main} we induct on the length of a maximal chain of $U^0(\AG)$-invariant subgraphs to construct a  marked $\G$-complex that realizes $G$.

\subsection{Acknowledgements} We would like to thank Dawid Kielak for several helpful discussions  {and the referee for his thoughtful comments}.  The first author is supported by NSF grant DMS-2052801.

\section{Review of RAAGs and the untwisted subgroup of $\Out(\AG)$.} Let $\G$ be a finite simplicial graph. The {\em right-angled Artin group (RAAG)} $\AG$ is the group generated by the vertices $V$ of $\G$ with defining relations given by declaring that adjacent vertices commute.

 By a subgraph of $\G$ we will always mean a full (induced) subgraph, unless otherwise specified. Given a subgraph $\Delta\subseteq \G$, we write $x\in\Delta$ if $x$ is a vertex of $\Delta$.

 For $x\in\G$, the {\em link} $\lk(x)$ is the subgraph spanned by vertices adjacent to $x$. The link of a subgraph $\Delta\subset \Gamma$ is the intersection of the links of all vertices of $ \Delta$. The {\em double link} $\dlk(x)$ is the link of $\Delta=\lk(x)$.

Recall from \cite{CSV} that a {\em $\G$-Whitehead partition} $\WP$ based at $x\in\Gamma$ is a partition of $V^\pm=V\cup V^{-1}$ into three sets: $\lk^\pm(x), P_1$ and $P_2$ satisfying certain conditions. The sets $P_1$ and $P_2$  are called the {\em sides} of $\WP$. A $\G$-Whitehead  partition can be most easily described using the {\em double} $\G^\pm$ of $\G$, where
the vertices of $\G^\pm$ are $V^\pm,$ and two vertices are joined by an edge if they commute but are not inverses   of each other. If $\WP$ is based at $x$,  $\lk^\pm(x)$ consists of all vertices adjacent to $x$ in $\G^\pm$, and each of $P_1,P_2$ is a union of (the vertices in) some connected components of $\Gamma^\pm\setminus \lk^\pm(x).$ Furthermore, we require $x $ and $ x\inv$  to be in different sides of $\WP$, and each side must contain at least one additional element. In this paper we will abbreviate  $\G$-Whitehead partition  to simply {\em $\G$-partition}.

A vertex $y \in V$ is \emph{split} by a $\G$-partition $\WP$ if $y$ is in one side and $y\inv$ is in the other.  If $y$ is split by $\WP$, then  $y$ and $y\inv$ must lie in different components of $\Gamma^\pm \backslash \lk^\pm(x)$, hence $\lk(y) \subseteq \lk(x)$.

A \GW partition $\WP$ based at $x$ determines  a {\em Whitehead automorphism} $\varphi(\WP,x)$,  {defined as follows. Let $P_i$ be the side of $\WP$ containing $x$.    If $\WP$ splits $y$ then $\varphi(\WP,x)$ sends $y\to yx^{-1}$ if $y\in P_i$, and $y\to xy$ if $y\inv\in P_i$.  If both $y$ and $y\inv$ are in $P_i$, then $\varphi(\WP,x)(y)=xyx\inv$.  For all other $y$,  $\varphi(\WP,x)(y)=y$. } 
The simplest Whitehead automorphisms are the {\em folds} sending $y\to yx\inv$  or $y\to xy$ for some $x$ and $y$ (and fixing all generators other than $y$), and the {\em partial conjugations} sending $y\to xyx^{-1}$ for all $y$ in some component $C$ of $\Gamma\setminus \lk(x)$.  These correspond to partitions $\WP=(\lk^\pm(x)|P_1|P_2)$ with $P_1=\{x,y\}$ or $\{x,y\inv\}$ (for a fold) or $P_1=\{x,C^\pm\}$ (for a partial conjugation). Every Whitehead automorphism is a product of folds and partial conjugations.

The subgroup of $\Out(\AG)$ generated by Whitehead automorphisms and by inversions of the generators is denoted $U^0(\AG)$.  If $\AG=F_n$ this is the whole group, \emph{i.e.} $U^0(F_n)=\Out(F_n)$.  If $\AG=\Z^n$ there are no Whitehead automorphisms, and $\Out(\Z^n)=\GL(n,\Z)$ is generated by inversions and  {\em twists}, where a twist sends a generator $y$ to $xy,$  for some $y$ with $\st(y)\subseteq st(x)$ and fixes all other generators.  By a theorem of Laurence and Servatius \cite{Lau,Ser}, for a general RAAG the group   $\Out(\AG)$ is generated by Whitehead automorphisms, inversions, twists and automorphisms of $\Gamma$.

The subgroup generated by $U^0(\AG)$ and graph automorphisms was called the {\em untwisted subgroup} and denoted $U(\AG)$ in \cite{CSV} and \cite{BrCV}.

\section{Blowups and $\G$-complexes}\label{sec:blowups}

Let $\WP=\big(\lk^\pm(x)|P_1|P_2\big)$ be a \GW partition based at $x$.  If $\lk(x)=\lk(y)$ and $\WP$ splits $y$,  then $y$ can also serve as a base for $\WP$.    Specifying a choice of base specifies the corresponding Whitehead automorphism, but we will often use  \GW partitions without specifying a base, in which case we write $\WP=\big(\lk(\WP)|P_1|P_2\big).$

We say  \GW partitions $\WP,\WQ$ are {\em adjacent} if some (hence any) base of $\WP$ commutes with some (any) base of $\WQ$, and they are {\em compatible} if either they are adjacent or some side of $\WQ$ is disjoint from some side of $\WP$.  A collection $\bPi$ of \GW partitions is a {\em compatible collection} if its elements are distinct and pairwise compatible.
In   \cite{CSV} the authors constructed a labeled cube complex $\SP$ called a {\em blowup}  from a compatible collection $\Pi$ of \GW partitions. The underlying (unlabeled) cube complex  is called  a {\em $\G$-complex.}
In this section we review some facts about special cube complexes and $\G$-complexes that we will need.

\subsection{Special cube complexes, collapsing and  duplicating   hyperplanes.}
Recall that a cube complex is called a {\em special cube complex} if it is locally CAT(0) and  has no hyperplanes that self-intersect or are one-sided, self-osculating or  inter-osculating. We refer to the original article by Haglund and Wise \cite{HaWi} for the basic definitions.

Let $X$ be a special cube complex.
If  $H$ is a hyperplane in $X$ the {\em collapse map} $c\colon X\to X\sslash H$ collapses the carrier $\kappa(H)$ of $H$ orthogonally onto $H$.   We say the result $X\sslash H$ is obtained from $X$ by a {\em hyperplane collapse} (See Figure~\ref{fig:collapse}).

\begin{figure}
  \begin{center}
 \begin{tikzpicture} [scale=.55]
\fill [red!15] (0,0) to (2,0) to (3,1) to (7,1) to (8,2) to (8,4) to (6,4) to (5,3) to (1,3) to (0,2) to (0,0);
  \fill [red!50] (0,.9) to (1,1.9) to (3, 1.9) to (2,.9) to (0,.9);
   \fill [red!50] (5,1.9) to (6,2.9) to (8, 2.9) to (7,1.9) to (5,1.9);
  \draw [thick,red] (3,1.9) to (5, 1.9);
   \draw (5,1) to (7,1) to (7,3) to (5,3) to (5,1);
  \draw (6,2) to (8,2) to (8,4) to (6,4) to (6,2);
    \draw (5,1) to (6,2);  \draw (7,1) to (8,2);  \draw (7,3) to (8,4);  \draw (5,3) to (6,4);
    \draw [blue] (0,2) to (0,4) to (2,4) to (2,2);
   \draw  [blue](1,3) to (1,5) to (3,5) to (3,3);
   \draw  [blue] (0,4) to (1,5);  \draw  [blue](2,4) to (3,5);
\draw [blue] (0,0) to (0,-2) to (2,-2) to (2,0);
   \draw  [blue](1,1) to (1,-1) to (3,-1) to (3,1);
   \draw  [blue] (0,-2) to (1,-1);  \draw  [blue](2,-2) to (3,-1);
  \draw (3,1) to (3,3) to (5,3) to (5,1) to (3,1);
  \draw (0,0) to (2,0) to (2,2) to (0,2) to (0,0);
   \draw (1,1) to (3,1) to (3,3) to (1,3) to (1,1);
   \draw (0,0) to (1,1);  \draw (2,0) to (3,1);  \draw (2,2) to (3,3);  \draw (0,2) to (1,3);
  \draw [blue] (6,4) to (6,6) to (8,6) to (8,4);
  \draw [blue] (5,1) to (5,-1) to (7,-1) to (7,1);
\draw [->] (9, 2.4) to (11,2.4);
 \begin{scope}[xshift=12cm]
  \draw [fill=red!50] (0,2) to (1,3) to (3, 3) to (2,2) to (0,2);
  \draw [fill=red!50] (5,3) to (6,4) to (8, 4) to (7,3) to (5,3);
  \draw [thick, red] (3,3) to (5, 3);
\draw [blue] (5,1) to (7,1) to (7,3) to (5,3) to (5,1);
    \draw [blue] (0,2) to (0,4) to (2,4) to (2,2);
   \draw  [blue](1,3) to (1,5) to (3,5) to (3,3);
   \draw  [blue] (0,4) to (1,5);  \draw  [blue](2,4) to (3,5);
 \draw  [blue] (0,0) to (2,0) to (2,2) to (0,2) to (0,0);
   \draw  [blue] (1,1) to (3,1) to (3,3) to (1,3) to (1,1);
   \draw  [blue] (0,0) to (1,1);  \draw (2,0) to (3,1);  \draw (2,2) to (3,3);  \draw (0,2) to (1,3);
  \draw [blue] (6,4) to (6,6) to (8,6) to (8,4);
  \end{scope}
\end{tikzpicture}
\end{center}
\caption{Hyperplane collapse $X\to X\sslash H$}~\label{fig:collapse}
\end{figure}

The edges that intersect a hyperplane $H$ are said to be {\em dual} to $H$, and by an {\em orientation} on $H$ we mean a consistent choice of orientation of the edges dual to $H$.

If $\mathcal S$ is a collection of hyperplanes, we write $X\sslash\mathcal{S}$ for the space obtained by collapsing all hyperplanes in $\mathcal{S}$  (in any order).  The collection $\mathcal S$ is   {\em acyclic}  if the collapse map $X\to X\sslash\mathcal{S}$ is a homotopy equivalence.

If $H$ is a hyperplane in $X$ with carrier $\kappa(H)$, we can obtain a new cube complex by  doubling $\kappa(H)$ (see Figure~\ref{fig:split}).  We will refer to this as    \emph{duplicating the hyperplane} $H$.    The resulting cube complex has two new hyperplanes $H'$ and $H''$, and collapsing either recovers the original complex $X$.  We say $H'$ and $H''$ are {\em parallel}.
 A hyperplane is called a  {\em duplicate} if it is parallel to another hyperplane.

\begin{figure} 
  \begin{center}
 \begin{tikzpicture} [scale=.55]
\fill [red!15] (0,0) to (2,0) to (3,1) to (7,1) to (8,2) to (8,4) to (6,4) to (5,3) to (1,3) to (0,2) to (0,0);
  \fill [red!50] (0,.9) to (1,1.9) to (3, 1.9) to (2,.9) to (0,.9);
   \fill [red!50] (5,1.9) to (6,2.9) to (8, 2.9) to (7,1.9) to (5,1.9);
  \draw [thick,red] (3,1.9) to (5, 1.9);
   \draw (5,1) to (7,1) to (7,3) to (5,3) to (5,1);
  \draw (6,2) to (8,2) to (8,4) to (6,4) to (6,2);
    \draw (5,1) to (6,2);  \draw (7,1) to (8,2);  \draw (7,3) to (8,4);  \draw (5,3) to (6,4);
    \draw [blue] (0,2) to (0,4) to (2,4) to (2,2);
   \draw  [blue](1,3) to (1,5) to (3,5) to (3,3);
   \draw  [blue] (0,4) to (1,5);  \draw  [blue](2,4) to (3,5);
\draw [blue] (0,0) to (0,-2) to (2,-2) to (2,0);
   \draw  [blue](1,1) to (1,-1) to (3,-1) to (3,1);
   \draw  [blue] (0,-2) to (1,-1);  \draw  [blue](2,-2) to (3,-1);
  \draw (3,1) to (3,3) to (5,3) to (5,1) to (3,1);
  \draw (0,0) to (2,0) to (2,2) to (0,2) to (0,0);
   \draw (1,1) to (3,1) to (3,3) to (1,3) to (1,1);
   \draw (0,0) to (1,1);  \draw (2,0) to (3,1);  \draw (2,2) to (3,3);  \draw (0,2) to (1,3);
  \draw [blue] (6,4) to (6,6) to (8,6) to (8,4);
  \draw [blue] (5,1) to (5,-1) to (7,-1) to (7,1);
\draw [->] (9, 2.4) to (11,2.4);
 \begin{scope}[xshift=12cm]
\fill [red!15] (0,0) to (2,0) to (3,1) to (7,1) to (8,2) to (8,4) to (6,4) to (5,3) to (1,3) to (0,2) to (0,0);
\fill [red!15] (0,2) to (2,2) to (3,3) to (7,3) to (8,4) to (8,6) to (6,6) to (5,5) to (1,5) to (0,4) to (0,2);
  \fill [red!50] (0,.9) to (1,1.9) to (3, 1.9) to (2,.9) to (0,.9);
   \fill [red!50] (5,1.9) to (6,2.9) to (8, 2.9) to (7,1.9) to (5,1.9);
  \draw [thick,red] (3,1.9) to (5, 1.9);
    \fill [red!50] (0,2.9) to (1,3.9) to (3, 3.9) to (2,2.9) to (0,2.9);
   \fill [red!50] (5,3.9) to (6,4.9) to (8, 4.9) to (7,3.9) to (5,3.9);
  \draw [thick,red] (3,3.9) to (5, 3.9);
   \draw (5,1) to (7,1) to (7,3) to (5,3) to (5,1);
  \draw (6,2) to (8,2) to (8,4) to (6,4) to (6,2);
    \draw (5,1) to (6,2);  \draw (7,1) to (8,2);  \draw (7,3) to (8,4);  \draw (5,3) to (6,4);
   \draw (5,3) to (7,3) to (7,5) to (5,5) to (5,3);
  \draw (6,4) to (8,4) to (8,6) to (6,6) to (6,4);
    \draw (5,3) to (6,4);  \draw (7,3) to (8,4);  \draw (7,5) to (8,6);  \draw (5,5) to (6,6);
    \draw [blue] (0,4) to (0,6) to (2,6) to (2,4);
   \draw  [blue](1,5) to (1,7) to (3,7) to (3,5);
   \draw  [blue] (0,6) to (1,7);  \draw  [blue](2,6) to (3,7);
\draw [blue] (0,0) to (0,-2) to (2,-2) to (2,0);
   \draw  [blue](1,1) to (1,-1) to (3,-1) to (3,1);
   \draw  [blue] (0,-2) to (1,-1);  \draw  [blue](2,-2) to (3,-1);
  \draw (3,1) to (3,3) to (5,3) to (5,1) to (3,1);
  \draw (3,3) to (3,5) to (5,5) to (5,3) to (3,3);
  \draw (0,0) to (2,0) to (2,2) to (0,2) to (0,0);
   \draw (1,1) to (3,1) to (3,3) to (1,3) to (1,1);
   \draw (0,0) to (1,1);  \draw (2,0) to (3,1);  \draw (2,2) to (3,3);  \draw (0,2) to (1,3);
  \draw (0,2) to (2,2) to (2,4) to (0,4) to (0,2);
   \draw (1,3) to (3,3) to (3,5) to (1,5) to (1,3);
   \draw (0,2) to (1,3);  \draw (2,2) to (3,3);  \draw (2,4) to (3,5);  \draw (0,4) to (1,5);
  \draw [blue] (6,6) to (6,8) to (8,8) to (8,6);
  \draw [blue] (5,1) to (5,-1) to (7,-1) to (7,1);
  \end{scope}
\end{tikzpicture}
\end{center}
\caption{Duplicating a hyperplane}~\label{fig:split}
\end{figure}
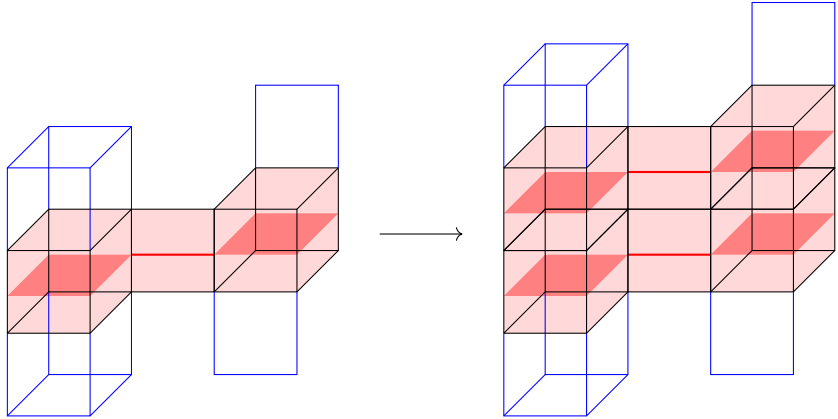
\subsection{Blowups} The blowup $\SP$ associated to  a compatible collection $\Pi$ of \GW partitions  is a special cube complex with no separating hyperplanes and with some extra structure.

If $\G$ is discrete then $\SP$ is a finite connected graph with no separating edges or bivalent vertices, and the extra structure consists of a maximal tree $T$ and an orientation and label on each edge  in $\SP\setminus T$, where the labels are the vertices of $\G$. Each edge of $T$ corresponds to a partition in $\Pi$, determined by the labels and orientations of the edges not in $T$.

  If $\bPi$ is empty, then $\SP$ is the {\em Salvetti complex} $\Sa_\G$ associated to $\AG$. Recall that this is a cube complex with a single $0$-cell, one oriented $1$-cell for each vertex of $\G$ and one $k$-torus for each $k$-clique in $\G$. The orientations on the $1$-cells, which are labeled by vertices of $\G$,  determine an isomorphism $\pi_1(\Sa_\G)\iso \AG$, and the cubical isomorphisms of $\Sa_\G$ can be identified with the automorphisms of the graph $\G$.

In general  $\SP$ has
\begin{itemize}
\item one hyperplane $H_\WP$ for each partition $\WP\in\Pi$ and
\item one hyperplane $H_v$ for each vertex $v\in\G$.
\end{itemize}
The hyperplanes $H_v$ are oriented,  but the hyperplanes $H_\WP$ are not.

The set of hyperplanes   labeled by partitions  is acyclic, and the complex obtained by collapsing all hyperplanes in this set is isomorphic to the Salvetti complex $\SaG$.

 Collapsing a single hyperplane labeled by $\WP\in \Pi$ is equivalent to removing $\WP$ from the collection $\Pi$. In particular collapsing every hyperplane in $\Pi$ other than $\WP$ results in a {\em single blowup} $\mathbb S^\WP$.  This has exactly two vertices $x_1$ and $x_2$, and one can recover $\WP=(\lk(\WP)|P_1|P_2)$ from $\mathbb S^\WP$ by looking at the (oriented!) edges dual to the hyperplanes  $H_v$:  if there is only one edge dual to  $H_v$ and it terminates at $x_i$, then $v\in P_i$; if it originates at $x_i$ then $v\inv\in P_i$.  If there are two edges dual to $H_v$ then $v\in \lk(\WP)$.  The  carrier of $H_\WP$ in $\mathbb S^\WP$ is isomorphic to the product of an interval with the Salvetti for $\lk(\WP)$.
 \eject
\subsection{ {$\G$-complexes}}
\begin{definition} A   cube complex is  called a {\em $\G$-complex}
if it is isomorphic to the underlying cube complex of a blowup $\SP$.   A {\em blowup structure} on a $\G$-complex $X$ is a labeling of its hyperplanes  that identifies $X$ with  a blowup $\SP$, {\em i.e.} hyperplanes are labeled by $\G$-partitions or by vertices of $\G$, and   the hyperplanes labeled by vertices are oriented. A blowup structure determines a {\em collapse map} $c_\pi\colon X\to \SaG$ that collapses all hyperplanes labeled by partitions.  If $v\in \G$, a {\em characteristic cycle} for $v$  is a closed edge path which crosses each hyperplane at most once, and whose image under $c_\pi$ is the loop labeled $v$.
\end{definition}

 In general a $\G$-complex may have several different blowup structures.  For example, if $\G$ is discrete a $\G$-complex is a graph, which  may have several different maximal trees, and the remaining edges may be oriented and labeled with the vertices of $\G$ in any way.

\begin{definition} A set $\mathcal T$ of hyperplanes in a $\G$-complex is called {\em treelike} if  collapsing $\mathcal T$ gives a cube complex isomorphic to $\SaG$.
\end{definition}

The following proposition says that any treelike set of hyperplanes in a $\G$-complex is the set of hyperplanes labeled by partitions in at least one blowup structure.   The only ambiguity comes from the assignment of labels and orientations to the hyperplanes not in the treelike set, which can be permuted by any automorphism of the graph $\G$. If this assignment changes, the partitions labeling the hyperplanes in the tree also change, by the same (signed) permutation of   vertices.

\begin{proposition}\label{lem:procedure} Let $\mathcal T$ be a treelike set of hyperplanes in a $\G$-complex $X$.  Then there is a compatible set of \GW partitions $\Pi$ and an isometry $X\iso \SP$, such that $\mathcal T$ is the set of hyperplanes associated to the partitions in $\Pi$.
\end{proposition}
\begin{proof}   We recall the construction.  For complete details see Section 4 of ~\cite{CSV}.

Label the edges dual to each hyperplane $H\in \mathcal T$ by  $H$.  Choose an isomorphism of $X\sslash\mathcal T$ with $\SaG$; this orients each hyperplane that is not in $\mathcal T$ and labels its dual edges by a vertex of $\G$.  The set of cubes in $X$ with all edge-labels in $\mathcal T$ forms a  CAT(0) subcomplex $\mathbb C$  that contains all vertices of $X$.  A hyperplane $H \in \mathcal T$ cuts $\mathbb C$ into two pieces, so partitions the vertices of $X$ into two sets,  $v_1(H)$ and $v_2(H)$. Now form a partition  $(\lk(H)| U_1|U_2)$ of $V^\pm$ as follows:
\begin{enumerate}
\item If the hyperplane $H_v$ labeled by $v$ intersects $H$, then $v$ and $v\inv$ are in $\lk(H).$
\item If $H_v\cap H =\emptyset$ and the terminal vertex of an edge dual to $H_v$ is in $v_i(H)$, then $v\in U_i$.
\item If $H_v\cap H=\emptyset$ and the initial vertex of an edge dual to $H_v$ is in $v_i(H)$ then $v\inv\in U_i$.
\end{enumerate}
Then the partition $(\lk(H)|U_1|U_2)$ is a \GW partition,   the set of \GW partitions for all $H\in\mathcal T$ is a    compatible collection $\Pi$, and $X$ is isomorphic to $\Sa_\G^\Pi$.
\end{proof}

We note that Condition (1) in the proof of Proposition~\ref{lem:procedure} is equivalent to saying that, in the universal cover $\widetilde \SP$, some lift of $H$ contains an axis for $v$. Saying that $v$ and $v\inv$ are in different $P_i$ is equivalent to saying that an axis for $v$ in $\widetilde \SP$ intersects some lift of $H$ transversally; in this case we say the axis  {\em skewers} $H$.    Saying $v$ and $v\inv$ are in the same $P_i$ is equivalent to saying that no axis for $v$ intersects any lift of $H$.

If we are given a special cube complex $X$ which we do not know a priori is a $\G$-complex, then to prove that  it is   we first need to find   an acyclic collection $\mathcal T$ of  hyperplanes which collapses to give  a cube complex isomorphic to $\SaG$. Choosing an isomorphism $X\sslash \mathcal{T} \cong \SaG$, gives  a labeling and orientation on all of the remaining hyperplanes.  We then need to check  that each hyperplane   $H\in \mathcal T$ determines a \GW partition.  We can do this   by collapsing all hyperplanes other than $H$ to get a complex with two vertices,   then checking whether  the location of the initial and terminal vertices of edges labeled by  $v\in\G$  gives a valid \GW partition.  {By Proposition~\ref{lem:procedure}, the partitions for one treelike set $\mathcal{T}$ are all \GW-partitions if and only if this holds for every treelike set.}

\subsection{Subdividing blowups}
In a blowup $\SP$ no two hyperplanes are parallel, \emph{i.e.} there are no  duplicate hyperplanes.   However, in the arguments that follow we will need to allow cubical subdivisions of blowups that result in  duplicate  hyperplanes.    Duplicating $H_\WP$  can be thought of as subdiviing its carrier $\kappa(H_\WP),$ and is  equivalent to adding a duplicate copy of $\WP$ to $\Pi$.  We want both of the  new hyperplanes we have created to be in the treelike set since we must collapse both to recover $\SaG$. Subdividing the carrier of $H_v$ is a little subtler; here we want only one of the two new hyperplanes to be added to the treelike set, so that collapsing the treelike set still gives $\SaG$.  In other words, when we duplicate $H_v$, we want  one of the two resulting hyperplanes to be labeled $H_v$, and the other to correspond to a partition. We also need the new $H_v$ to have the orientation induced from the old $H_v$.  This is accomplished by adding a ``singleton partition" to $\Pi$; this is a partition based at $v$ with one side containing only $v$ (if we want the initial segment of the dual edge to retain the $v$ label) or $v\inv$ (if we want the terminal segment to retain the $v$ label). To make this   a canonical operation, we can consistently use the singleton partition $\mathcal S_{v\inv}=(\lk^\pm(v)|\{v\inv\}|(V\setminus\lk(v))^\pm\setminus\{v\inv\})$, so that the terminal segment always retains the $v$ label.

Note that duplicate partitions fit the definition of ``compatible with each other,'' and a singleton partition is compatible with every \GW partition.   A set of pairwise compatible partitions that is allowed to have singletons and duplicates will be called a   {\em compatible multi-set}.    By the above remarks,   compatible multi-sets correspond to subdivided blowups.

\section{$U^0$-invariant subgraphs}\label{sec:invariant}\label{sec:U0invariant}

Recall that a {\em marking} on a $\G$-complex is a homotopy equivalence $h\colon X\to \SaG$.
Let $G$ be a finite group and $\rho\colon G\to \Out(\AG)$  a homomorphism.

\begin{definition} A marked $\G$-complex $(X,h)$ {\em realizes}  $\rho$  if  there is an action   $f\colon G\to \Aut(X)$ of $G$ on $X$ by cubical automorphisms  such that $h \circ f(g)\circ h\inv$ induces $\rho(g)$ on $\pi_1(\SaG)=\AG$ for all $g\in G$.
\end{definition}

Our goal in this paper is to build a marked $\G$-complex that realizes (the inclusion of) a finite subgroup $G < U^0(\AG)$.  Our approach is inductive.  Specifically, we will build our $\G$-complex by gluing together marked $\Delta$-complexes for subgraphs $\Delta$   which are {\em $U^0$-invariant}, in the sense that elements of $U^0(\AG)$ preserve the special subgroup $A_\Delta$ (up to conjugacy).  We will induct on the length of a maximal chain of $U^0$-invariant subgraphs.  In this section we prepare for this by establishing some basic facts about   $U^0$-invariant  subgraphs.

\begin{definition} Given $\phi\in \Out(\AG)$,  a subgraph $\Delta$ of $\Gamma$ is {\em  $\phi$-invariant}  if $\widehat\phi(A_\Delta)$ is conjugate to $A_\Delta$ for some lift $\widehat\phi$ of $\phi$ to $\Aut(\AG)$. Since any two such lifts differ by an inner automorphism, this is well-defined. A subgraph $\Delta$ is  {\em $U^0$-invariant} if it is $\phi$-invariant for every $\phi\in U^0(\AG)$.\end{definition}

 The next lemma gives a criterion for $U^0$-invariance.

\begin{lemma}\label{lem:InvariantSubgraph} Let $\Delta$ be a subgraph of $\Gamma$.  Then $\Delta$ is $U^0$-invariant if and only if the following two conditions hold for all $x,y\in \G$:
\begin{enumerate}
\item[(i)] if $x\in  \Delta$ and $\lk(x)\subseteq \lk(y)$ then $y\in \Delta$ and
\item [(ii)] if  $\Delta$ intersects more than one component of $\Gamma\setminus \st(y)$ then $y\in\Delta$.
\end{enumerate}
\end{lemma}
\begin{proof}  Since every subgraph $\Delta\subseteq \G$ is invariant under inversions, a subgraph $\Delta$ is $U^0$-invariant if and only if  it is invariant under the remaining generators of $U^0(\AG)$, {\em i.e.} all folds and partial conjugations. There is a fold  $\tau\colon x\mapsto xy$ if and only if $\lk(y)\supseteq \lk(x)$, so $\tau$ maps $x\in  \Delta$ to $A_\Delta$ if and only if  $y\in \Delta$.
If $\Delta$ intersects two different components $C$ and $C'$ of $\Gamma\setminus \st(y)$, then $A_\Delta$ is  sent to a conjugate of itself under the partial conjugation $C\mapsto yCy\inv$ if and only if $y\in\Delta$.
\end{proof}

\begin{proposition}\label{cor:applications} Let $\G$ be a simplicial graph.
\begin{enumerate}
\item If $\Sigma$ is a subgraph of $\Gamma$, then $\Delta=\lk(\Sigma)$ is $U^0$-invariant.
\item If $\Delta_1$ and $\Delta_2$ are two $U^0$-invariant subgraphs of $\Gamma$, then $\Delta_1\cap\Delta_2$ is $U^0$-invariant.
\item If $\Delta_1$ and $\Delta_2$ are two $U^0$-invariant subgraphs of $\Gamma$ whose join $\Delta_1*\Delta_2$ is also a subgraph, then $\Delta_1*\Delta_2$ is $U^0$-invariant.
\item  If $\Sigma$ is a non-singleton connected component of a $U^0$-invariant subgraph $\Delta$, then $\Sigma$ is $U^0$-invariant.
\item  If $\Delta$ is $U^0$-invariant and $N(\Delta)$ is the subgraph spanned by $\Delta$ and all vertices adjacent to  $\Delta$, then $N(\Delta)$ is $U^0$-invariant.
\end{enumerate}

\end{proposition}
\begin{proof} In each case we check conditions $(i)$ and $(ii)$ of Lemma~\ref{lem:InvariantSubgraph}:

Proof of (1)
\begin{itemize}
\item[(i)]   $x\in\Delta=\lk(\Sigma)$ if and only if $\Sigma\subseteq \lk(x)$.  If $\lk(x)\subseteq\lk(y)$ then $\Sigma\subset \lk(y)$ so    $y\in\lk(\Sigma)$.
\item[(ii)] Suppose that $\Delta=\lk(\Sigma)$ intersects two different components $C_1$ and $C_2$  of $\Gamma\setminus\st(y),$ say $x_1\in C_1\cap\Delta$ and $x_2\in C_2\cap\Delta$.  If $z\in\Sigma$ then $x_1$ and $x_2$ are both connected to $z$, so $z$ must be in $\lk(y)$.  Thus $\Sigma\subset \lk(y)$ so $y\in\lk(\Sigma)$.
\end{itemize}

Proof of (2)
\begin{itemize}
\item[(i)] Let $x\in\Delta_1\cap\Delta_2$.  If $\lk(y)\supseteq \lk(x)$ then $y\in\Delta_1$ by invariance of $\Delta_1$; similarly $y\in\Delta_2$, so $y$ is in the intersection.
\item[(ii)] If $\Delta_1\cap\Delta_2$ intersects two components of $\Gamma\setminus\st(y)$ then the same is true of both $\Delta_1$ and $\Delta_2$, so $y$ is in both.
\end{itemize}

Proof of (3)
\begin{itemize}
\item[(i)] If $x\in\Delta_1$ and $\lk(y)\supseteq\lk(x)$ then $y\in\Delta_1$; similarly if $x\in\Delta_2$.
\item[(ii)]  {Suppose $\Delta=\Delta_1*\Delta_2$ intersects two different components $C,C'$ of $\Gamma\setminus\st(y)$.  Let $x\in C\cap\Delta$ and $x'\in C'\cap\Delta$.  Then $x$ and $x'$ must be in the same $\Delta_i$, since otherwise there is an edge connecting them. But $\Delta_i$ is $U^0$-invariant, so $y\in \Delta_i\subset \Delta$. }  
\end{itemize}

Proof of (4)
\begin{itemize}
\item[(i)]   Suppose $x\in\Sigma$, and  $\lk(y)\supset\lk(x)$.  Since $\Sigma$ is not a singleton, $\lk(x)\cap \Sigma$ contains a vertex $z$.  Since $z$ is in the links of both  $x$ and $y$, $y$ is also in $\Sigma$.
\item[(ii)]  Suppose $x,z\in \Sigma$ are in different components of $\G\setminus \st(y)$.  Since $\Delta$ is $U^0$-invariant, this implies $y\in\Delta$, but then $\st(y)$ cannot separate $x$ from $z$ unless $y\in\Sigma$.
\end{itemize}

Proof of (5)
\begin{itemize}
\item[(i)] If $x\in N(\Delta)$ then either $x\in \Delta$ or $\lk(x)\cap \Delta\neq \emptyset$. If the distance from $y$ to $\Delta$ is at least 2, then $\lk(y)\cap \Delta=\emptyset$, so $\lk(x)\not \subseteq \lk(y)$.
\item[(ii)]   {If $y\not\in N(\Delta)$ then $\Delta\cap\, \st(y)=\emptyset$.} Since $\Delta$ is $U^0$-invariant, it  lies in a single component of $\G\setminus \st(y)$, so  {all} $x\in N(\Delta)\setminus \st(y)$ must lie in the same component.
\end{itemize}
\end{proof}

Recall from \cite{BrCV} that two vertices are called {\em fold eqivalent} if they have the same link in $\G$, and we order the set of fold equivalence classes by inclusion  of their links.  The following two propositions will allow us to establish the base case of our induction.

\begin{proposition}\label{lem:MinimalInvariant}  Let $\Delta\subseteq\Gamma$ be a
minimal $U^0$-invariant subgraph. Then $\Delta$ is a maximal fold equivalence class. In particular, $\Delta$ is discrete.
\end{proposition}
\begin{proof} If $[u]$ is a maximal fold equivalence class then   it is easy to check using Lemma~\ref{lem:InvariantSubgraph} that $[u]$ is $U^0$-invariant. On the other hand, if $\Delta\subseteq \G$ is any $U^0$-invariant subgraph then $\Delta$ contains a maximal equivalence class $[u]$ by condition (1) of Lemma~\ref{lem:InvariantSubgraph}, so  $[u]=\Delta$ by the minimality of $\Delta$. \
\end{proof}

 \begin{proposition}\label{prop:restriction} Let $\Delta$ be a $U^0$-invariant subgraph of $\G$.  Then there is a restriction homomorphism $r_\Delta\colon U^0(A_\G)\to U^0(A_\Delta)$.
\end{proposition}
\begin{proof} Let $\phi$ be an element of $U^0(\AG)$.  Since $\Delta$ is $U^0$-invariant there is a lift $\widehat\phi$  of $\phi$ to $\Aut^0(\AG)$ with $\widehat\phi(A_\Delta)=A_\Delta$.  Define $r_\Delta(\phi)$ to be the image in $\Out(A_\Delta)$ of the restriction of $\widehat\phi$ to $A_\Delta$.

To check that $r_\Delta$ is well-defined, suppose $\widehat\phi'$ is another lift  of $\phi$ sending $A_\Delta$ to itself.  Then $\widehat\phi'=\iota_g\circ\phi$ where $\iota_g$ is conjugation by some $g\in\AG$ that normalizes $A_\Delta$.    By Lemma 2.2 of ~\cite{CCV}  the normalizer of $A_\Delta$ is   $A_\Delta\times A_{\lk(\Delta)}$.  Since elements of $A_{\lk(\Delta)}$ act trivially by conjugation, $\widehat\phi'=\iota_h\circ \widehat\phi$ for some $h\in A_\Delta$, \emph{i.e.}   the images of $\widehat\phi$ and $\widehat\phi'$  in $\Out(A_\Delta)$ are equal.  Moreover, if $\phi_1,\phi_2$ are two elements of $U^0(\AG)$ and $\phi = \phi_1 \circ \phi_2$, then $\widehat\phi = \widehat\phi_1 \circ \widehat\phi_2$ is a lift of $\phi$ which preserves $A_\Delta$.  Thus  $r_\Delta\colon U^0(A_\G)\to \Out(A_\Delta)$ is a homomorphism.

To see that this lands in $U^0(A_\Delta)$, we check that this is the case for the generators of $U^0(\AG)$.  Let $\phi$ be a generator that lifts to a fold $\widehat \phi\colon x \mapsto xy$, so $\lk(x)\subset\lk(y)$.   If $x \notin \Delta$, then $\widehat\phi$  restricts to the identity on $A_\Delta$.   If $x \in \Delta$, then $U^0$-invariance of $\Delta$ implies that $y$ is also in $\Delta$, so the restriction is the lift of a fold in $U^0(A_\Delta)$.
Now suppose $\hat\phi$ is a partial conjugation by $x$.   If $x \notin \Delta$, then the fact that $\Delta$ is $U^0$-invariant, implies that $\Delta$ lies entirely in one component of $\Gamma \backslash \st_\Gamma(x)$, so the action of $
\widehat\phi$ to  $A_\Delta$ is trivial.   If $x \in \Delta$ then $\st_\Delta(x)  \subseteq \st_\Gamma(x)$ so the components of $\Delta \backslash \st_\Delta(x)$ are contained in components of $\Gamma \backslash \st_\Gamma(x)$.  Thus  $\widehat\phi$ restricts to a (product of) partial conjugations by $x$ on $A_\Delta$.
\end{proof}

\begin{definition} If $\Delta$ is a $U^0$-invariant subgraph  and $f\colon H\to U^0(A_\G)$ is any homomorphism, we call $f_\Delta=r_\Delta\circ f\colon H\to U^0(A_\Delta)$ the {\em restriction} of $f$ to $\Delta$.
\end{definition}

\section{ $U^0$-invariant subcomplexes of  marked  $\G$-complexes}\label{sec:XDelta}

Throughout this section we assume $\Delta\subseteq \G$ is a $U^0$-invariant subgraph  with $\lk(\Delta)=\emptyset$  and $(X,h)$ is $\G$-complex with an untwisted marking.  This means that for any  blowup structure $\Sa_\G^\Pi$ on $X$ with associated collapse map $c_\pi\colon  X\to \Sa_\G,$  the composition $c_\pi h\inv $ induces an untwisted automorphism of $\AG$,  that is, an element of $U(\AG)$.  The aim  is to identify a  subcomplex $X_\Delta\subset X$ which is invariant under the action of isometries that induce elements of $U^0(\AG)$.  {The reason for the restriction that $\lk(\Delta)=\emptyset$ is that in this case the subcomplex $X_\Delta$ is unique.   Remark~\ref{rmk:nonemptylink} discusses  the general case. }

The following lemma deals with the discrepancy between the subgroups $U(\AG)$ and $U^0(\AG)$.

\begin{lemma} If $X$ is a $\G$-complex with an untwisted marking $h$, then $X$ has a blowup structure $\SP$ with collapse map $c_\pi\colon X=\SP\to \SaG$ such that $(c_\pi h\inv )_*\in U^0(\AG)$.
\end{lemma}
\begin{proof}  $U^0(\AG)$ is normal in the untwisted subgroup $U(\AG)$   of $\Out(\AG)$, and the quotient $Q$ is a subgroup of $\Aut(\G)$,  the group of graph automorphisms of $\G$.   The short exact sequence
$$1\to U^0(\AG)\to U(\AG)\to Q\to 1$$
splits, so $U(\AG)=U^0(\AG)  \rtimes Q$.

Let $\SP$ be any blowup structure on $X$ such that  $(c_\pi h\inv )_*$ is untwisted. By the above observation, if $(c_\pi h\inv )_*$ is not in $U^0(\AG)$ we can compose it with a graph automorphism $\alpha$ to produce an element of $U^0(\AG)$.  Realize $\alpha$ by an isometry $f_\alpha$ of $\Sa_\G$.
Composing $c_\pi$ with $f_\alpha$ is equivalent to changing the labels and orientations of the vertex-labeled hyperplanes in $\SP$ by $\alpha$; this changes the partitions by the same relabeling, giving a new set of partitions $\alpha\Pi$. In other words, this gives a blowup structure $\Sa_\G^{\alpha\Pi}$ on $X$ so that $(h\inv c_{\alpha\Pi})_*\in U^0(\AG).$
\end{proof}

Now fix a blowup structure $\SP$ on $X$ such that $(c_\pi h\inv )_* \in U^0(\AG)$. In any \GW partition all bases have the same link.  Since $\Delta$ is $U^0$-invariant, Lemma~\ref{lem:InvariantSubgraph} implies that either all bases of a partition in $\Pi$ are  in $\Delta$, or none are. So we may write $\bPi=\{\WQ_1,\ldots,\WQ_k,\WP_1,\ldots,\WP_\ell\}$, where the $\WQ_i$ are based in $\Delta$ and the $\WP_j$ are based in $\G\setminus\Delta$.

\begin{lemma} Each $\WP_i$ has a unique side $P_i^\Delta$ such that $\Delta^\pm\subset P^\Delta_i\cup\lk(\WP)$ and $\Delta^\pm\cap P^\Delta_i\neq \emptyset$.  Furthermore, if $\WP_i$ and $\WP_j$ are not adjacent, then $P_i^\Delta\cap P_j^\Delta\neq \emptyset$.
\end{lemma}
\begin{proof} Let $y_i\in\G\setminus \Delta$ be a base for $\WP_i$.  Recall that we have assumed $\lk(\Delta)$ is empty; this implies that $\Delta\setminus \st(y_i)$ is nonempty.    Since $\Delta$ is $U^0$-invariant,  it intersects at most one component of $\Gamma\setminus \st(y_i)$.
Each side of $\WP_i$ is a union of components of $\Gamma\setminus\st(y_j)$. Thus, there is a unique side $P_i^\Delta$ such that  $\Delta^\pm\cap P_i^\Delta\neq \emptyset$ and $\Delta^\pm\subseteq P_i^\Delta \cup \lk(\WP_i)^\pm$.

If $\WP_i$ and $\WP_j$ are compatible   but not adjacent and  $P_i^\Delta\cap P_j^\Delta=\emptyset$ then $P_i^\Delta\cap \lk(\WP_j)=P_j^\Delta\cap \lk(\WP_i)=\emptyset$ as well (see Lemma 2.9 of \cite{BrCV}), forcing $\Delta^\pm\subset \lk(\WP_i)\cap\lk(\WP_j)$.  This contradicts our assumption that $\Delta^\pm\cap P_i^\Delta\neq\emptyset.$
\end{proof}

 Recall from~\cite{CSV} that vertices of $\Sa_\G^\Pi$ correspond to collections $\{Q_1^\times,\ldots,Q_k^\times,P_1^\times,\cdots, P_\ell^\times\}$, where the superscript $\times$ indicates a choice of side.  Each pair of sides in the collection must be {\em consistent},  which means either they intersect non-trivially or their bases  commute.
We define $\mathbb K_\Delta$ to be the subcomplex of $\SP$ consisting of vertices of the form $\{Q_1^\times,\ldots,Q_k^\times,P_1^\Delta,\cdots, P_\ell^\Delta\}$, edges obtained by switching a side of some $\WQ_i$ or labeled by some $v\in\Delta$, and all higher-dimensional cubes spanned by these edges.

\begin{figure}
  \begin{tikzpicture}[scale=.8]
   \begin{scope}[decoration={markings,mark = at position 0.5 with {\arrow{stealth}}}]
  \coordinate (c) at (2,1); \coordinate (d) at (2,-1); \coordinate (a) at (1,0);\coordinate (b) at (3,0);  \coordinate (e) at (-.5,0); \coordinate (f) at (4.5,0);
\fill [blue!15] (2,0) ellipse (1.5cm and 1.5cm);
 \vertex{(a)};\vertex{(b)};\vertex{(c)};\vertex{(d)};\vertex{(e)};\vertex{(f)};
  \coordinate (ci) at (-1,-0); \coordinate (ai) at (0,0); \coordinate (di) at (1,0);\coordinate (bi) at (2,0);
 \node [above left](A) at (a) {$a$};
  \node [above right](B) at (b) {$b$};
    \node [above right](C) at (c) {$c$};
      \node [ right](D) at (d) {$d$};
      \node [above left](E) at (e) {$e$};
      \node [above right](F) at (f) {$f$};
      \node (Delta) at (2,0) {$\Delta$};
\draw [thick] (e) to (a) to (c) to (b) to (f); \draw[thick] (a) to (d) to (b);
\node at (2,-4) {$(a)$};
\begin{scope}[xshift = 2cm]
\draw[rounded corners, densely dotted] (5,2.8) to (5,.25) to (7,.25) to (7,2.8) to  (5,2.8) to (5,.5)--cycle;
 \coordinate (ap) at (5.5,2); \coordinate (bp) at (6.5,2); \coordinate (cp) at (8,2);\coordinate (dp) at (9,2);  \coordinate (ep) at (10,2); \coordinate (fp) at (11,2);
 \vertex{(ap)};\vertex{(bp)};\vertex{(cp)};\vertex{(dp)};\vertex{(ep)};\vertex{(fp)};
  \node [above](A) at (ap) {$a$};
  \node [above](B) at (bp) {$b$};
    \node [above](C) at (cp) {$c$};
      \node [above](D) at (dp) {$d$};
      \node [above](E) at (ep) {$e$};
      \node [above](F) at (fp) {$f$};
  \coordinate (am) at (5.5,1); \coordinate (bm) at (6.5,1); \coordinate (cm) at (8,1);\coordinate (dm) at (9,1);  \coordinate (em) at (10,1); \coordinate (fm) at (11,1);
   \vertex{(am)};\vertex{(bm)};\vertex{(cm)};\vertex{(dm)};\vertex{(em)};\vertex{(fm)};
     \node [below](Ai) at (am) {$\bar a$};
  \node [below](Bi) at (bm) {$\bar b$};
    \node [below](Ci) at (cm) {$\bar c$};
      \node [below](Di) at (dm) {$\bar d$};
      \node [below](Ei) at (em) {$\bar e$};
      \node [below](Fi) at (fm) {$\bar f$};
\draw [rounded corners, blue] (8.75,1.75) to (8.75,2.4) to (10.25,2.4) to (10.25,1.75) to (8.75,1.75)--cycle; 
\draw [rounded corners, blue]  (7.75,2.5) to (7.75,1.6) to (9.6,1.6) to (9.6,.75 ) to (10.35,.75) to (10.35,2.5) to (7.75,2.5) to (7.75,1.6)--cycle; 
\draw [rounded corners, blue]  (7.5,2.65) to (7.5,1.45) to (9.3,1.45) to (9.3,.6 ) to (10.55,.6) to (10.55,1.6) to (11.25,1.6) to (11.25,2.65) to (7.5,2.65) to (7.5,1.45)--cycle;
\draw [rounded corners, blue]  (7.25,2.8) to (7.25,1.3) to (8.5,1.3) to (8.5,.45 ) to  (10.7,.45) to  (10.7,1.45) to  (11.4,1.45) to (11.4,2.8) to (7.25,2.8) to (7.25,1.3)--cycle; 
\end{scope}
\begin{scope}[xshift=2cm, yshift=-3cm]
\draw[rounded corners, densely dotted] (5,2.8) to (5,.25) to (7,.25) to (7,2.8) to  (5,2.8) to (5,.5)--cycle;
\coordinate (ap) at (5.5,2); \coordinate (bp) at (6.5,2); \coordinate (cp) at (8,2);\coordinate (dp) at (9,2);  \coordinate (ep) at (10,2); \coordinate (fp) at (11,2);
\vertex{(ap)};\vertex{(bp)};\vertex{(cp)};\vertex{(dp)};
  \node [above](A) at (ap) {$a$};
  \node [above](B) at (bp) {$b$};
    \node [above](C) at (cp) {$c$};
      \node [above](D) at (dp) {$d$};
  \coordinate (am) at (5.5,1); \coordinate (bm) at (6.5,1); \coordinate (cm) at (8,1);\coordinate (dm) at (9,1);  \coordinate (em) at (10,1); \coordinate (fm) at (11,1);
   \vertex{(am)};\vertex{(bm)};\vertex{(cm)};\vertex{(dm)};
     \node [below](Ai) at (am) {$\bar a$};
  \node [below](Bi) at (bm) {$\bar b$};
    \node [below](Ci) at (cm) {$\bar c$};
      \node [below](Di) at (dm) {$\bar d$};
\draw [rounded corners, blue] (8.75,1.75) to (8.75,2.4) to (9.25,2.4) to (9.25,1.75) to (8.75,1.75)--cycle;
\draw [rounded corners, blue]  (7.75,2.5) to (7.75,1.6) to (9.5,1.6)  to (9.5,2.5) to (7.75,2.5) to (7.75,1.6)--cycle;
\draw [rounded corners, blue]  (7.5,2.65) to (7.5,1.45) to (9.65,1.45) to (9.65,2.65) to (7.5,2.65) to (7.5,1.45)--cycle;
\draw [rounded corners, blue]  (7.25,2.8) to (7.25,1.3) to (8.5,1.3) to (8.5,.45 ) to  (9.8,.45) to (9.8,2.8) to (7.25,2.8) to (7.25,1.3)--cycle;
\end{scope}
\node at (10,-4) {$(b)$};
\end{scope}\end{tikzpicture}
\caption{(a) A graph $\Gamma$ and $U^0$-invariant subgraph $\Delta$. (b) Four different $\G$-partitions with the same link and their restrictions to $\Delta^\pm$. In the restriction, the innermost and outermost partitions become  trivial and the middle two become equal.}\label{fig:partitions}
\end{figure}
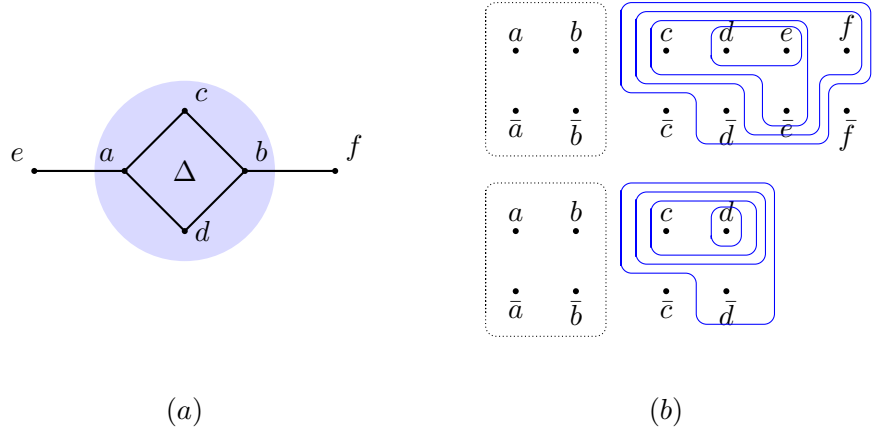

 \begin{proposition}\label{prop:LinklessSubcomplex}
  The subcomplex   $\K_\Delta$  is a subdivided blowup of $\Sa_\Delta$.
 \end{proposition}

\begin{proof} Let $\{\WQ_1,\ldots,\WQ_k\}$ be the partitions of $\bPi$ based in $\Delta$,  let $\Omega$   be  the multi-set of $\Delta$-partitions   obtained by intersecting the $\WQ_i$  with $\Delta^\pm$ (see Figure~\ref{fig:partitions}), and let $\Sa_\Delta^{\Omega}$ be the corresponding subdivided blowup.   {If the bases of $\WQ_i$ and $\WQ_j$ don't commute, then exactly one pair of   sides has empty intersection (see \cite{CSV}, Lemma 3.6).  The $\WQ_i \cap \Delta^\pm$ are $\Delta$-partitions, and the corresponding pair of  sides still has empty intersection.  This means that they are compatible, and all other pairs of sides must intersect.}  Since the $\WQ_i$ are based in $\Delta$,  this means that $\{Q_1^\times,\ldots,Q_k^\times\}$ is consistent in $\G^\pm$ if and only if $\{ Q_{1}^\times\cap\Delta^\pm,\ldots, Q_{k}^\times\cap\Delta^\pm\}$ is consistent in $\Delta^\pm$, \emph{i.e.} defines a vertex of $\Sa_\Delta^{\Omega}$.  Note that for each $j=1,\ldots k$, either $\WP_i$ is adjacent to $\WQ_j$ or $P_i^\Delta$ intersects both sides of  $\WQ_j$ nontrivially, since $\WQ_j$ splits its own base, which is in $\Delta\setminus \lk(\WP_i)$. Therefore the map sending $$\{ Q_{1}^\times\cap\Delta^\pm,\ldots, Q_{k}^\times\cap\Delta^\pm\}  \to
 \{Q_1^\times,\ldots,Q_k^\times,P_1^\Delta,\cdots, P_\ell^\Delta\}$$ is well-defined, and induces an isomorphism of $\Sa_\Delta^{\Omega_\Delta}$ with $\mathbb K_\Delta$.
\end{proof}

A priori our marking $h\colon X\to\SaG$  maps $\K_\Delta$ to $\Sa_\G$, but we can adjust it within its homotopy class to map $\K_\Delta$ to  $\Sa_\Delta$.  This is because $h=u\circ c_\pi$, where $u:\SaG\to\SaG$ induces an element $\phi\in U^0(\AG)$.  We can choose a representative $\hat\phi\in\Aut(\AG)$  that sends $A_\Delta$ to itself, so adjusting $u$ by a homotopy we get  a map sending $\Sa_\Delta$ to itself.  Since $c_\pi$ sends $\K_\Delta$ to $\Sa_\Delta$, we may assume the composition $h=u\circ c_\pi$ restricts to a marking on $\K_\Delta$.

We constructed the subcomplex $\K_{\Delta}\subset X$ using the blowup structure  that we chose on $X$.
We now  show that $\K_{\Delta}$ is  essentially  independent of this choice.

 \begin{proposition}\label{prop:XDelta}  The subcomplex $\K_\Delta$ is independent of the blowup structure $\Sa_\G^\Pi$ as long as this blowup structure satisfies  $(c_\pi h\inv)_*\in U^0(\AG)$.
   \end{proposition}

Proposition~\ref{prop:XDelta} can be proved combinatorially by keeping track of how the partitions change when we change the treelike set of hyperplanes or the labelings and orientations on the  hyperplanes not in the treelike set.  Being based in $\Delta$ turns out to be a property of hyperplanes in the treelike set, independent of the partitions used to describe them.  The same is true for the property of being the $\Delta$-side of a hyperplane   based in $\G\setminus\Delta$.  Since these properties are what is used to define $\K_\Delta$, the subcomplex $\K_\Delta$ itself is independent of the   {blowup} structure. Furthermore, an isometry $f\colon X\to X$ satisfying $(hfh\inv)_*\in U^0(\AG)$ preserves these properties, so preserves $\K_\Delta$.

  Below we give a different proof of Proposition~\ref{prop:XDelta}  in terms of the action of $\AG$ on $\UX$ determined by $h$.  This proof is more in the spirit of our previous paper \cite{BrCV} and more amenable to generalization.  We continue to assume that $\lk(\Delta)=\emptyset$.

The universal cover $\widetilde X$ is a CAT(0) cube complex, and we will take advantage of   the following facts about isometries of CAT(0) cube complexes.
We  say that a hyperbolic automorphism $g$ of a CAT(0) cube complex  {\em skewers} a hyperplane $H$ if some axis for $g$ crosses $H$ transversely.

\begin{lemma}\label{lem:skewer} Let $g$ be a hyperbolic automorphism of a CAT(0) cube complex  {$\widetilde X$} and let $H$ be a hyperplane in   {$\widetilde X$}.  If some axis for $g$ skewers $H$, then every axis for $g$ skewers $H.$ If no axis for $g$ skewers $H$,  then either all axes for $g$ are on the same side of $H$, or  $H$ contains an axis for $g$.
\end{lemma}

\begin{proof} Let $\alpha_1, \alpha_2$ be two axes for $g$.  Suppose $\alpha_1$ crosses $H$ but not $\alpha_2$ does not.  Let $H^+$ denote the half space containing $\alpha_2$ and $H^-$ the complementary half space.  Set $\alpha_1^\pm= H^\pm \cap  \alpha_1$.   Let $\gamma$ be a geodesic connecting a point in $\alpha_1^-$ to a point in $\alpha_2$.  Any such path must cross $H$.

The action of $g$ preserves both axes so either $g$ or $g^{-1}$ maps $\alpha_1^+$ into itself.  Without loss of generality, assume $g(\alpha_1^+) \subset \alpha_1^+$ , or equivalently, $\alpha_1^- \subset g(\alpha_1^-)$.  Consider the action on  $\UX$  by positive powers of $g$.  Since $H$ does not intersect $\alpha_2$, the hyperplanes $g^kH$ also do not intersect $\alpha_2$.  Thus $g^kH$ separates  $\alpha_1^-$ from $\alpha_2^+$ for all $k>0$.  But this means that the path $\gamma$ must cross infinitely many hyperplanes, which is impossible.

 {Now suppose $\alpha_1, \alpha_2$ lie on opposite sides of $H$.  The minset of g decomposes as an orthogonal product of an axis $\alpha$ and a convex subspace of $\widetilde X$. It follows that $\alpha_1, \alpha_2$ span a strip $\alpha \times I$  for some interval $I$.  This strip intersects $H$ in a convex set which separates these two axes.  Any such set must contain $\alpha \times y$ for some point $y \in I$.}
\end{proof}

\begin{proof}[Proof of Proposition~\ref{prop:XDelta}]  Fix a basepoint at a vertex $x_0\in X$, and let $p_0$ denote the base vertex of $\SaG$. We may assume the marking $h\colon X\to \Sa_\G$ sends $x_0$ to $p_0$, so induces an isomorphism $\pi_1(X,x_0)\iso\pi_1(\SaG,p_0)\equiv\AG$.  The collapse map $c_\pi\colon X=\SP\to \SaG$ gives another marking, and by construction the composition $\phi=(h\circ c_\pi^{-1})_*\colon A_\G\rightarrow A_\G$ lies in $U^0(\AG)$ (as an automorphism). Since $A_\Delta$ is $U^0$-invariant, $\phi$ sends $A_\Delta$ to a conjugate of itself. We may therefore homotope $h$, by dragging the basepoint $p_0$ around a loop in $\SaG$, so that $\phi$ sends $A_\Delta$ isomorphically to $A_\Delta$. If we choose a lift $\tilde x_0$ of $x_0$ to the universal cover $\tilde X$, we obtain two actions of $\AG$ on $\tilde X$, one from $h$ and the other from $c_\pi$. Since $\phi$ preserves $A_\Delta$, this means that the axes for elements of $A_\Delta$ under both actions coincide setwise.

Our goal will now be to characterize $\K_\Delta$ as subcomplex of $X$  purely in terms of the $U^0$-marking $h\colon X\rightarrow \SaG$, or equivalently, in terms of the action of $\AG$ on $\tilde X$. The preceding paragraph implies that if we can characterize $\K_\Delta$ in terms of the set of axes of elements of $A_\Delta$, it does not matter whether we use the action from $h$ or $c_\Pi$.

Now let $\mathcal H$ be the set of hyperplanes in $\UX$ that  are {\em not} skewered by  any element of $A_\Delta$.  By Lemma~\ref{lem:skewer} this is the same set whether we are considering the action defined by $h$ or by $c_\pi$.
\begin{claim*}  If $H\in\mathcal H$, then exactly one half-space $H^+\subset \UX$  contains an axis for every element of $A_\Delta$.
\end{claim*}

\begin{proof}[Proof of claim] Let $\widetilde\K_\Delta$ be the lift of $\K_\Delta$ preserved by the action of $A_\Delta$.  The lift $\widetilde\K_\Delta$ is convex, hence is connected and contains an axis for every element of $A_\Delta$.

Let $H$ be a hyperplane in $\UX$. We claim that $H\in\mathcal H$ if and only if $H\cap \widetilde \K_\Delta=\emptyset$, so that all of $\widetilde \K_\Delta$ is on the same side of $H$, and that side contains an axis for every $g\in A_\Delta$.

If $H\cap \widetilde \K_\Delta\neq\emptyset$ then $\widetilde \K_\Delta$ contains an edge dual to $H$.  Lemma 3.10 of \cite{BrCV} implies that every edge in $\widetilde \K_\Delta$ is in some axis for some $g\in A_\Delta$, so $H$ is skewered by an element of $A_\Delta,$ \emph{i.e.} $H$ is not in $\mathcal H$.  Conversely, if $H$ is skewered by an element $g\in A_\Delta$ then by Lemma~\ref{lem:skewer} every axis for $g$ skewers $H$, so in particular some axis contained in $\widetilde \K_\Delta$ skewers $H$, so $H$ intersects $\widetilde \K_\Delta$.

If both half-spaces determined by $H$ contain axes for every element of $A_\Delta$ then by
Lemma~\ref{lem:skewer}, $H$ itself contains an axis for every element of $A_\Delta$.  Let $e_H$ be an edge dual to $H$.  Then $e_H$ is contained in an axis for some element $w \in A_\G$, since that is true of every edge in $\UX$,  and $w \notin A_\Delta$ since $w$ skewers $H$.  By Lemma 3.10 of \cite{BrCV}, all of $H$ is contained in the min set for $w$ so $w$ commutes with every element of $A_\Delta$, \emph{i.e.} $A_\Delta$ has nontrivial centralizer.  This contradicts the assumption that  $\lk(\Delta)= \emptyset$.
\end{proof}

 We now define
 $$\UX_\Delta=\bigcap_{H\in\mathcal H} H_\Delta,$$
 where $H_\Delta=H^+\setminus \kappa(H)$ is the largest subcomplex of $\UX$ contained in $H^+$.  This is independent of the blowup structure, and coincides with $\widetilde{\K}_\Delta$ for the action defined by $c_\pi$. Therefore the image $\K_\Delta$ of $\widetilde X_\Delta$ in $X$ is independent of the blowup structure.
 \end{proof}
 
 { \begin{remark} In the terminology of \cite{CapraceSageev}, the set of hyperplanes $\mathcal{H}$ occurring in the proof of Proposition~\ref{prop:XDelta} are exactly those which are \emph{inessential} for the action of $A_{\Delta}$.  Indeed, by \cite[Proposition 3.2(ii)]{CapraceSageev}, the essential core $\textrm{Ess}(\widetilde X,A_\Delta)$ for the action of $A_\Delta$ on $\widetilde X$ consists of those hyperplanes skewered by the axis of $A_\Delta$, \emph{i.e.} the complement of $\mathcal{H}$. Thus, the proposition asserts the existence of a convex subcomplex of $\widetilde{X}$ whose hyperplanes extend exactly to $\textrm{Ess}(\widetilde X,A_\Delta)$, and on which $A_\Delta$ acts cocompactly.  
 \end{remark}}

{\em Notation.} Since Proposition~\ref{prop:XDelta} shows that $\K_\Delta$ is independent of the blowup structure, we emphasize this by using the notation $X_\Delta$ instead of $\K_\Delta$ for the image of $\widetilde X_\Delta$ in $X$.

\begin{corollary}\label{cor:NoLinkRestriction}  Let $\Delta$ be a $U^0$-invariant subgraph with $\lk(\Delta)=\emptyset$, and $G<U^0(\AG)$ a finite subgroup which is realized on a  $\G$-complex $X$ with an untwisted marking $h\colon X\to\SaG$.   Then the restriction of $G$ is realized on  $(X_\Delta,h|_{X_\Delta})$. \end{corollary}

\begin{proof}  An element of $G$ is realized by an isometry $f\colon X\to X$.  An isometry  of $X$ sends any blowup structure to a new blowup structure.  If the isometry induces an element of $U^0(\AG)$, we have shown that $X_\Delta$ has not changed, so $f$ must send $X_\Delta$ to itself.
\end{proof}

\begin{remark}\label{rmk:nonemptylink}  We constructed $X_\Delta$ assuming that $\lk(\Delta)=\emptyset$.   If this is not the case, we can look instead at $\st(\Delta),$ which always has empty link.  We will see in Section~\ref{sec:joins} that the complex for $\st(\Delta)=\Delta*\lk(\Delta)$ breaks into a product $X_\Delta\times X_{\lk(\Delta)}$, where $X_\Delta$ and $X_{\lk(\Delta)}$ are    $\Delta$- and $\lk(\Delta)$-complexes respectively.
\end{remark}

\section{Extendable $\Delta$-complexes}\label{sec:extendable}

Let $\Delta$ be a $U^0$-invariant subgraph of $\G$,   $\rho\colon G\to U^0(\AG)$ a homomorphism from a finite group $G$  and $(X,h)$ a marked $\G$-complex realizing $\rho$.  In the last section we found a $G$-invariant $\Delta$-complex  (possibly subdivided) sitting inside $X$.  In this section we consider the opposite problem: given a marked $\Delta$-complex $Y$ realizing the restriction $\rho_\Delta\colon G\to U^0(\AG)$, when can $Y$ sit equivariantly inside  a marked $\G$-complex?     If there is such a marked $\G$-complex we say $Y$ is {\em extendable}.
 To determine when $Y$ is extendable,  we first define what it means for  a $\Delta$-partition to be extendable.

 \begin{definition}Let $\Delta\subset \G$ be a $U^0$-invariant subgraph and let $\WQ$ be a $\Delta$-partition.  We say that $\WQ$ is \emph{extendable} if there exists a $\G$-partition $ \widehat\WQ$ such that $ \widehat\WQ\cap \Delta^\pm=\WQ$.
\end{definition}

\begin{proposition}\label{lem:extendable} Let $\Delta$ be a $U^0$-invariant subgraph of $\G$ and $\mathcal Q$ a $\Delta$-partition.  Then $\mathcal Q$ is extendable if and only if  {there is some base $m$ of $\WQ$ such that} 
\begin{enumerate}
\item  $\lk_\G(v)\subseteq \lk_\G(m)$ for every $v$ split by $\WQ$, and
\item
if $v_1$ and $v_2$ are in the same component of $\G\setminus\st_\G(m)$, then $v_1^\pm,v_2^\pm$ are all in the same side of $\WQ$.
\end{enumerate}
\end{proposition}
\begin{proof}
The ``only if'' direction is immediate, since any extension of $\WQ$ is a $\G$-partition.   Note that any extension of $\WQ$ also splits $m$.   In fact it has to be based at $m$  since   $\Delta$ is $U^0$-invariant, which implies there is no $v\in \G\setminus \Delta$ with $\lk(v)\supseteq \lk(m)$.

For the converse, suppose $\WQ=(\lk_\Delta^\pm(m)|Q_1|Q_2)$ satisfies conditions (1) and (2).  We   build  a $\G$-partition $(\lk^\pm_\G(m)|\widehat Q_1|\widehat Q_2)$ based at $m$ as follows.

 { If $v\in\lk_\Delta(m)$ then $v\in\lk_\G(m)$. If $v$ (resp. $v\inv$) is  in $Q_i$  put $v$ (resp. $v\inv$) in $\widehat Q_i$. This determines where to place all $v^{\pm 1}$ for $v\in\Delta$.}
 
 { Now suppose $v\in\G\setminus (\Delta\cup\st_\G(m))$ and let $C$ be the component of $\G\setminus \st_\G(m)$ containing $v$. If $C\cap\Delta=\emptyset$, put all vertices of $C$ and their inverses in the same side of $\widehat \WQ$ (either side will do).  
If $C\cap\Delta$ is non-empty then some side $Q_i$ of $\WQ$ contains an element $w\in \Delta$.  We must have $w\inv\in Q_i$ as well, since otherwise $\lk(w)\subseteq\lk(m)$ by condition (1), which would imply that $w$ was the only vertex of $C$. By condition (2) the side $Q_i$ is independent of the choice of $w$. Put all vertices of $C$ and their inverses into $\widehat Q_i$.  }  
\end{proof}

 \begin{definition}
  A blowup $\mathbb S_\Delta^\Omega$ is {\em extendable} if every $\WQ\in\Omega$ is extendable. Note we are not assuming the extended partitions are compatible. A $\Delta$-complex is {\em extendable} if it can be given the structure of an extendable blowup.
 \end{definition}

\begin{remark} It is not hard to show that if a $\Delta$-complex is extendable with respect to one blowup structure, then it is extendable with respect to any blowup structure, but we will not need this fact.
\end{remark}

\subsection{$U^0$-invariant subgraphs and extendability}

In this subsection we give a condition that guarantees that a $\Delta$-complex realizing $\rho_\Delta$ is extendable.

 \begin{definition}  A $G$-action on a $\G$-complex $X$ is {\em  reduced} if no orbit of hyperplanes is contained in any treelike set. A marked $\G$-complex $(X,h)$ realizing $\rho \colon G\rightarrow \Out(A_\G)$   is {\em  reduced} if the associated $G$-action on $X$ is reduced.
\end{definition}

If a marked $\G$-complex $(X,h)$ realizing $\rho$ is not reduced, then  some orbit $G.H$ is acyclic since it is contained in the treelike (acyclic) set associated to some blowup structure on $X$. We can collapse every hyperplane in $G.H$ to produce a new marked $\G$-complex.
The following lemma guarantees that the new $\G$-complex still realizes $\rho$.

\begin{lemma}\label{lem:CollapseSep} Let   $X$ be an NPC cube complex and let $G\leq \Isom(X)$ be a subgroup. Suppose $\mathcal{S}$ is a collection of hyperplanes that is acyclic and $G$-invariant.  If $G\to \Out(\pi_1(X))$ is injective, then the collapse map $c\colon X\rightarrow X\sslash\mathcal{S}$ induces an injection $G\rightarrow \Isom(X\sslash\mathcal{S})$.
\end{lemma}
\begin{proof} Let $\pi=\pi_1(X)$. Since $\mathcal{S}$ is acyclic, $c$ is a homotopy equivalence. The fact that $G$ preserves $\mathcal{S}$ means there is an induced map $\bar{c}\colon G\rightarrow \Isom(X\sslash\mathcal{S})$.  We obtain a commutative diagram:
\begin{equation}\xymatrix{G\ar[r]\ar[d]^{\bar{c}} &\Out(\pi)\ar[d]^{\cong}\\
\Isom(X\sslash\mathcal{S})\ar[r]& \Out(\pi)
}\end{equation}
By assumption, $G\leq \Isom(X)$ injects into $\Out(\pi)$, hence $G$ also injects under $\bar{c}$.
\end{proof}

Let $(X,h)$ be a marked $\G$-complex realizing $\rho$.   By Lemma~\ref{lem:CollapseSep} we may continue to collapse $G$-orbits of hyperplanes until we obtain a reduced marked $\G$-complex realizing $\rho$.  Note that the result of this process is not unique, but depends on the set of orbits we choose to collapse.

 In Section~\ref{sec:XDelta} we produced a subcomplex $X_\Delta$ of  $X$  for any $U^0$-invariant subgraph $\Delta$ with empty link.  If $(X,h)$ is reduced,  then in any blowup structure the orbit of any hyperplane labelled by a partition $\WP$ contains a hyperplane labelled by an element of $V$.  Since the action of $G$ preserves the subcomplex $X_{\Delta}$, the same must be true for orbits in $X_{\Delta}$, so the action of $G$ on $X_{\Delta}$ is also reduced.
As an example, note that if $(X,h)$ is reduced and the restriction of $G$ to $U^0(A_{\Delta})$ is trivial, then the subcomplex  $X_{\Delta}$ must be equal to the Salvetti complex $\Sa_\Delta$.  This follows since no non-trivial blow-up of $\Sa_\Delta$ is  reduced with respect to the trivial action. Thus  reduced realizations of $G$ may be thought of as ``minimal" $\G$-complexes realizing $G$.

Let $\Delta$ be a $U^0$-invariant subgraph of $\G$.  The next proposition states that being reduced is sufficient to guarantee extendability for any  $\Delta$-complex realizing the restriction $\rho_\Delta$.

\begin{proposition}\label{prop:ExtendableAmalgam}Let $(X_\Delta,h_\Delta)$ be a marked $\Delta$-complex that realizes $\rho_\Delta$.  If the action of $G$ on $X_\Delta$ is  reduced, then $X_\Delta$ is extendable.
\end{proposition}

\begin{proof}  {Let $V$ be the set of vertices in $\Delta$.}
Choose a blowup structure $X_\Delta\cong \Sa_\Delta^\Omega$.   Since the $G$-action is reduced, for every $\WQ\in \Omega$ there exists $g\in G$ such that $g.H_\WQ=H_w$ for some $w\in V$. By Proposition~\ref{lem:extendable}, in order to verify extendability  we must show
\begin{enumerate}[(i)]
\item  There is some base $m$ of $\WQ$ such that $\lk_\G(v)\subseteq \lk_\G(m)$ for every $v$ split by $\WQ$.
\item
If $v_1$ and $v_2$ are in the same component of $\G\setminus\st_\G(m)$, then $v_1^\pm,v_2^\pm$ are all in the same side of $\WQ$.
\end{enumerate}

\textbf{Proof of (i).}
Let $\chi$ be a characteristic cycle for $v$ in $\Sa_\Delta^\Omega$.  The image  $g.\chi$   is a path that crosses each hyperplane of $\Sa_\Delta^\Omega$ at most once, so the edge-labels that are not in $\Omega$ spell a cyclic word that is the image of $v$ under $g$:
\begin{equation}\label{eqn:CycleImage} g(v)= x_1^{\epsilon_1} \cdots x_k^{\epsilon_k},\end{equation}
where all the $x_i$ are distinct and $\epsilon_i=\pm 1$.

\begin{claim}\label{claim:PreservingMax} If $x_i\in V$ labels an edge in $g.\chi$ then $\lk_\G(v)\subseteq \lk_\G(x_i)$. If $\WQ\in \Omega$ labels an edge in $g.\chi$ then $\WQ$ splits some $x_i$.
\end{claim}
\begin{proof}
 {The double link $\dlk(v)=\lk(\lk(v))$ is $U^0$-invariant by Proposition~\ref{cor:applications}(1).  Since $v\in\dlk(v)$ this implies that each $x_i$ appearing in Equation \ref{eqn:CycleImage} must be in $\dlk(v)$, \emph{i.e.} $\lk_\G(v)\subseteq\lk_\G(x_i)$.}
Now suppose $\WQ$ labels an edge of $g.\chi$, for some $\WQ\in \Omega$. We claim that $\WQ$ splits at least one of the $x_i$. Indeed, after choosing a basepoint $*\in X_\Delta$, we see that $g.\chi$ is freely homotopic to a concatenation of edge paths $\eta_1\eta_2\cdots \eta_k$  where $\eta_i$ is an edge-path based at $*$ representing the element $x_i$. Each $\eta_i$ is freely homotopic to a characteristic cycle for $x_i$, hence $\eta_i$ crosses the hyperplane $H_\WQ$ an odd number of times if and only if $\WQ$ splits $x_i$. Since $\eta_1\cdots \eta_k$ is freely homotopic to $g. \chi$, and $g.\chi$ crosses $H_\WQ$ exactly once, we must have that $\WQ$ splits some $x_i$.
\end{proof}

The proof of (i) now follows directly from the following claim.

\begin{claim}\label{claim:Well_Defined_Max} For every $\WQ\in \Omega$, there exists $m\in \Sing(\WQ)$ such that $\lk_\G(v)\subseteq \lk_\G(m)$ for every $v\in \Sing (\WQ)$. Moreover, defining $\lk_\G(\WQ)=\lk_\G(m)$,  the action of $G$ preserves $\lk_\G(E)$ for every edge label $E\in V\cup \Omega$.
\end{claim}
\begin{proof}
Define an increasing filtration $\emptyset=V_0 \subsetneq V_1\subsetneq \cdots \subsetneq V_r=V$ where  for $i\geq 1$, $V_{i}\setminus V_{i-1}$ consists of all $v\in \Delta$ such that $\lk_\G(v)$ is maximal among elements of $V\setminus V_{i-1}$. {Every partition $E$ in $\Omega$ splits some element of $V$,} so we can extend this to an increasing filtration $\emptyset=\F_0 \subsetneq \F_1\subsetneq \cdots \subsetneq \F_r=V\cup \Omega$, by letting $E\in\F_i$ if $H_E$ splits some generator of $V_i$. We will prove by induction that
\begin{enumerate}[(a)]
\item For every $\WQ\in \F_i$, there exists $m\in \Sing(\WQ)\cap V_i$ such that $\lk_\G(v)\subseteq \lk_\G(m)$ for every $v\in \Sing(\WQ)$.
\item Defining $\lk_\G(\WQ)=\lk_\G(m)$, then for any $A,B\in \F_i$, if $g.H_A=H_B$, then $\lk_\G(A)=\lk_\G(B)$.
\end{enumerate}

The base case $\F_0=\emptyset$ holds vacuously. Suppose by induction that for some $i\geq 1$ we have verified (a) and (b) for $\F_{i-1}$. Since the $G$-orbit of every $H_\WQ$ contains $H_w$ for some $w\in V$, it follows from (b) that $\F_{i-1}$ is the union of all $G$-orbits of elements of $V_{i-1}$. Consider now $\WQ \in \F_{i}\setminus \F_{i-1}$.  Then any generator in $\Sing(\WQ)$ lies in $V\setminus V_{i-1}$, and $\WQ$ splits some $m\in V_i\setminus V_{i-1}$. Let $g\in G$ be such that $g.H_\WQ=H_w$ for some $w\in V$. For any $v\in \Sing(\WQ)$ we know that  $\lk_\G(v)\subseteq \lk_\G(w)$  by Claim \ref{claim:PreservingMax}. In particular, $\lk_\G(m)\subseteq \lk_\G(w)$  and therefore $w\in V_{i}$. On the other hand, since $\F_{i-1}$ is a union of $G$-orbits and does not contain $\WQ$, we know that $w\notin V_{i-1}$.  Hence $w\in V_{i}\setminus V_{i-1}$ and therefore $\lk_\G(w)=\lk_\G(m)$ as $m$ is maximal among all elements of $V\setminus V_{i-1}$.  It follows that $\lk_\G(v)\subseteq\lk_\G(w)=\lk_\G(m)$ for any $v\in \Sing(\WQ)$, which proves (a).

Let $\chi$ be a characteristic cycle for $v\in V_i\setminus V_{i-1}$. Claim \ref{claim:PreservingMax} implies that any hyperplane crossed by $g. \chi$ has a $\G$-link that contains $\lk_\G(v)$, hence its label is in $\F_i$. Since $\F_{i-1}$ is a union of $G$-orbits, if some label appearing in $g.\chi$ is not in $\F_{i-1}$, it must be in $\F_i\setminus \F_{i-1}$, and therefore is equivalent to $v$. In particular, if $g.H_v=H_A$ then  $\lk_\G(A)=\lk_\G(v)$.  Since every label in $\F_i\setminus \F_{i-1}$ appears in such an orbit, we conclude that $G$ preserves the $\G$-link of each element of $\F_i$, which proves (b) and completes the inductive step.
\end{proof}

\textbf{Proof of (ii).}
Suppose $\WQ\in \Omega$ is based at $m$ and $v_1^\pm,v_2^\pm$ lie on opposite sides of $\WQ$.  Since $X_\Delta$ is a blowup of $\Sa_{\Delta}$, we know that $v_1,v_2$ lie in different components of $\Delta \setminus \st_\Delta(m)$.  We must show that they lie in different components of $\G \setminus \st_\G(m)$.

As shown in \cite{BrCV}, the inverse image of the (unique) vertex of $\Sa_\Delta$ under the collapse map $\Sa_\Delta^\Omega \to \Sa_\Delta$ is a CAT(0) subcomplex of $\Sa_\Delta^\Omega$, consisting of cubes whose edges are all labelled by partitions.  Denote this subcomplex by $\C^\Delta$.
Choose a minimal length edgepath $\alpha$ in $\C^\Delta$ between any characteristic cycle for $v_1$ and any characteristic cycle for $v_2$.  Then $\alpha$ crosses exactly those hyperplanes labeled by partitions containing $v_1^\pm$ and $v_2^\pm$ on opposite sides.  In particular, it crosses the hyperplane labeled $\WQ$. For $i=1,2$, let $\chi_i,$ be a characteristic cycle for $v_i$ starting at either end of $\alpha$, and consider the edgepath $\gamma=\chi_1\alpha\chi_2\bar{\alpha}$.  Under the collapse map to $\Sa_\Delta$, the loop $\gamma$ represents the element $v_1v_2$. Observe that by the minimality of $\alpha$, the hyperplanes crossed by $\alpha$, $\chi_1$ and $\chi_2$ are pairwise disjoint. Given an element $g\in G$, we have
\[g_*(\gamma)=g_*(\chi_1)g_*(\alpha)g_*(\chi_2)g_*(\overline{\alpha})=g_*(\chi_1)g_*(\alpha)g_*(\chi_2)\overline{g_*(\alpha)}\]
Since $\alpha$, $\chi_1$, and $\chi_2$ cross each hyperplane of $X_{\Delta}$ at most once, and cross pairwise distinct sets of hyperplanes, the same is true of $g_*(\alpha)$, $g_*(\chi_1)$, and $g_*(\chi_2)$. Therefore the hyperplanes crossed by $g_*(\gamma)$ that are not labeled by partitions define a (cyclic) word in the generators that is the image of $v_1v_2$ under the action of $g$. Thus we may write $g_*(v_1v_2)=w_1uw_2u^{-1}$
where \[u=y_1^{\delta_1}\cdots y_r^{\delta_r}\] is a word in pairwise distinct generators $y_j$ and $\delta_j=\pm1$.

Since $X_\Delta$ is reduced,   there exists some $g \in G$ such that $g$ maps the hyperplane labeled $\WP$ to a hyperplane labeled  by one of the $y_j$'s, and by Claim \ref{claim:Well_Defined_Max}, $m$ and $y_j$ belong to the same $\Gamma$-equivalence class.  It thus suffices to prove that $v_1$ and $v_2$ lie in different components of $\Gamma\setminus \st(y_j)$ for each $j$.

For any vertex $v$ in $\G$,  $\dlk(v)$ is $U^0$-invariant, up to conjugacy.  Thus, the cyclically reduced form of the word $g_*(v_1)$, namely $w_1$, must be a word in $\dlk(v_1)$  and similarly $w_2$ must be a word in $\dlk(v_2)$.  Choosing the representative of $g_*$ in $Aut(\AG)$  to be one that takes $v_1$ to $w_1$,  we then have $g_*(v_2)=uw_2u^{-1}$ where $u$ is a product of generators $y_i$ such that $\st(y_i)$ separates some element of $\dlk(v_1)$ from some element of $\dlk(v_2)$. But in this case, $\st(y_i)$ must also separate $v_1$ from $v_2$ as required.
\end{proof}

\section{ $\G$-complexes for joins and disjoint unions } \label{sec:joins}

In this section we describe procedures for constructing $\G$-complexes realizing $G$ when  $\G$ is either a join or a disjoint union of $U^0$-invariant subgraphs realizing the restriction of $G$.  We begin with the case where $\G=\G_1*\G_2$ is a join, which is straightforward.

\begin{proposition}\label{lem:ProductBlowup}If $\Gamma=\G_1*\G_2$ then any $\G$-complex   is a product of a $\G_1$-complex and a $\G_2$-complex.  {Conversely, any product of a $\G_1$-complex and a $\G_2$-complex is a $\G_1*\G_2 $-complex.}
\end{proposition}
\begin{proof}
Suppose $\Sa_\G^\Pi$ is a blowup structure on a $\G$-complex $X$.  Write  $$\bPi=\{\WP_1,\ldots, \WP_k,\WQ_1,\ldots, \WQ_\ell\},$$where $\WP_i$ is based at $x_i\in \G_1$ and $\WQ_j$ is based at $y_j\in \G_2$.   (Note that if one base is in $\G_i$ then all bases are in $\G_i$).  Since $\G=\G_1\ast\G_2$  every $x_i$ is adjacent to every $y_j$, intersecting each $\WP_i$ with $V^\pm(\G_1)$  gives a $\G_1$-partition $\WP_i^1$ and intersecting  each $\WQ_j$  with $V^\pm(\G_2)$ gives a $\G_2$-partition $\WQ_j^2$. Thus, $\bPi$ is a compatible collection of $\G$-partitions if and only if $\bPi_1=\{\WP_1^1,\ldots, \WP_k^1\}$ is a compatible collection of $\G_1$-partitions and $\bPi_2=\{\WQ_1^2,\ldots, \WQ_\ell^2\}$ is a compatible collection of $\G_2$-partitions.  We conclude that
$X$ is a product of the $\G_i$ complexes with blowup structures  $\Sa_{\Gamma_1}^{\bPi_1}$ and $\Sa_{\Gamma_2}^{\bPi_2}$ respectively.
\end{proof}

\begin{proposition}\label{lem:RealizingJoin} Suppose $\G=\G_1*\G_2$ and let $\rho\colon G\rightarrow U^0(\AG)$ be a homomorphism.  Then $\G_1$ and $\G_2$ are $U^0$-invariant and if $(X_{i},h_i)$, $i=1,2$ are marked $\G_i$-complexes  that realize the restrictions $\rho_i$ of $\rho$ to $\G_i$,  their product $X_\G=X_{1}\times X_{2}$, equipped with the product action of $G$ and the product marking $h=h_1\times h_2$,  is a marked $\G$-complex realizing $\rho$. Moreover, the action of $G$ on $X_\G$ given by $\rho$ is  reduced if and only if the actions on $X_1$ and $X_2$ are.
\end{proposition}
\begin{proof}
Observe that $\G_1=\lk(\G_2)$ and $\G_2=\lk(\G_1)$, hence by Proposition~\ref{cor:applications}, both are proper $U^0 $-invariant subgraphs.   For $i=1,2$, suppose $X_{i}$ is a $\G_i$-complex and that $h_i\colon X_{i}\rightarrow \Sa_{\G_i}$ is a marking which realizes the restriction $\rho_i$. Define $X_{\G}=X_{1}\times X_{2}$ and let $G$ act via the product action. Blowup structures on $X_1$ and $X_2$ give a blowup structure to $X_1\times X_2$  by Proposition \ref{lem:ProductBlowup}, and the product marking $h=h_1\times h_2\colon  X_{\G}\rightarrow \Sa_\G$ realizes the action of $G$. The final statement of the lemma follows from the fact that hyperplanes of $X_\G$ are all of the form $H_1\times X_2$ for $H_1$ a hyperplane of $X_1$ or $X_1\times H_2$ for $H_2$ a hyperplane of $X_2$, and that the action of $G$ preserves the product decomposition.
\end{proof}

We next consider the case when $\Gamma$ is a disjoint union of (not necessarily connected) subgraphs.   Given a $\G_i$-complex $X_i$ for each subgraph $\G_i$ that is not a singleton, we construct a $\G$-complex $X$ that contains each $X_i$ as a subcomplex.

\begin{definition}[$\G$-amalgam] Suppose $\Gamma$ is a disjoint union $\G_1\sqcup \cdots \sqcup \G_k\sqcup \Lambda,$ where $\Lambda$ is discrete.   Let $Z_\Lambda$ be a graph satisfying
\begin{itemize}
\item The rank of $Z_\Lambda$ is equal to $|\Lambda|$,
 \item  $k$ vertices of $Z_\Lambda$ are labeled by $\{1,\ldots, k\}$,
 \item Any unlabeled vertex of $Z_\Lambda$ has valence at least 3.
\end{itemize}
For each $i\in \{1,\ldots, k\}$, let  $X_{i}$ be a $\G_i$-complex.
Form a new cube complex $Y_\Gamma$ called a \emph{$\Gamma$-amalgam} as follows. For each vertex $v$ of $Z_\Lambda$, set $X_v=X_{i}$ if $v$ is labeled by $i$, and set $X_v$ to be a point otherwise. Now construct a complex by starting from the disjoint union $\sqcup X_v$ and attaching an edge from $X_v$ to $X_w$ whenever $\{v,w\}$ is an edge in $Z_\Lambda$.  When $X_v=X_{i}$ we may attach the edge anywhere. The resulting complex can be given the structure of a cube complex: if the endpoint of one of the added edges lies at a point $p$ in the interior of a cube $C\subseteq X_{i}$, we perform the cubical subdivision of $C$ at $p$.  Define the resulting cube complex to be $Y_\Gamma$. (See Figure~\ref{fig:disjoint}.)

\end{definition}

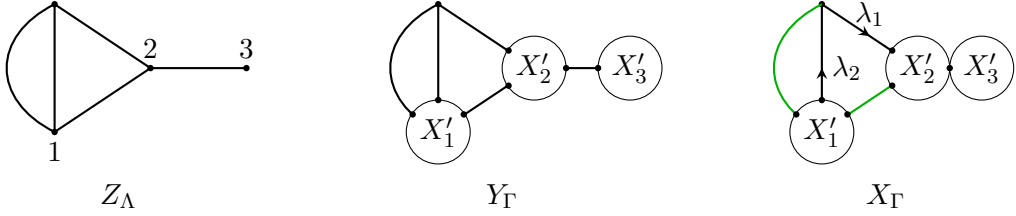
\begin{figure}\begin{center}
 \begin{tikzpicture} [scale=.85]
\begin{scope}[decoration={markings,mark = at position 0.5 with {\arrow{stealth}}}];
\vertex{(0,0)};\vertex{(0,2)};\vertex{(1.5,1)};\vertex{(3,1)};
  \draw[thick](0,0) to (0,2);
 \draw[thick] (0,0) to (1.5,1);
  \draw[thick] (0,2) to (1.5,1);
    \draw[thick] (1.5,1) to (3,1);
    \draw[thick] (0,0).. controls (-1,.5) and  (-1,1.5) .. (0,2);
 \node [below] (n1) at (0,0) {$1$};
        \node[above] (n2) at (1.5,1) {$2$};
            \node[above] (n3) at (3,1) {$3$};
            \node (Delta) at (1,-1) {$Z_\Lambda$};

    \begin{scope}[xshift = 6cm]:
 \vertex{(0,2)};
  \draw[thick] (0,0) to (0,2);
 \draw[thick] (0,0) to (1.5,1);
 \draw[thick] (0,2) to (1.5,1);
    \draw[thick] (1.5,1) to (3,1);
    \draw[thick] (0,0).. controls (-1,.5) and  (-1,1.5) .. (0,2);
 \draw[fill=white] (0,0) circle(.5);
  \draw[fill=white] (1.5,1) circle(.5);
    \draw[fill=white] (3,1) circle(.5);
    \node (S1) at (0,0) {$X'_1$};
        \node (S2) at (1.5,1) {$X'_2$};
            \node (S3) at (3,1) {$X'_3$};
             \node (Y) at (1,-1) {$Y_\Gamma$};
  \vertex{(0,.5)};\vertex{(2,1)};\vertex{(2.5,1)}; \vertex{(.4,.275)};\vertex{(-.4,.275)};\vertex{(1.1,1.275)};\vertex{(1.1,.725)};
            \end{scope}
                \begin{scope}[xshift = 12cm]:
                 \vertex{(0,2)};
\midarrow[thick] (0,0) to (0,2);\node[right] (l2) at (0,1) {$\lambda_2$};
 \draw[thick, darkgreen] (0,0) to (1.5,1);
  \midarrow[thick] (0,2) to (1.5,1);\node[above] (l1) at (.75,1.5) {$\lambda_1$};
    \draw[thick, darkgreen] (0,0).. controls (-1,.5) and  (-1,1.5) .. (0,2);
 \draw[fill=white] (0,0) circle(.5);
  \draw[fill=white] (1.5,1) circle(.5);
    \draw[fill=white] (2.5,1) circle(.5);
    \node (S1) at (0,0) {$X'_1$};
        \node (S2) at (1.5,1) {$X'_2$};
            \node (S3) at (2.5,1) {$X'_3$};
              \node (X) at (1,-1) {$X_\Gamma$};
              \vertex{(0,.5)};\vertex{(2,1)};; \vertex{(.4,.275)};\vertex{(-.4,.275)};\vertex{(1.1,1.275)};\vertex{(1.1,.725)};
            \end{scope}
         \end{scope}
 \end{tikzpicture}
 \end{center}\caption{Building a $\G$-complex for a disjoint union $\G=\G_1\sqcup \G_2\sqcup \G_3\sqcup\Lambda$.}\label{fig:disjoint}
 \end{figure}

Observe that $Y_\Gamma$ contains a subdivided copy  $X'_{i}$ of $X_{i}$ as a subcomplex for each each $i$.  Moreover, collapsing each of these subcomplexes separately to a point defines a natural map $ Y_\Gamma\rightarrow Z_\Lambda$, which is a bijection away from the $X_{i}$. Since each $X_{i}$ is a $\G_i$-complex, it does not have any separating hyperplanes. It follows that a hyperplane of $Y_\Gamma$ is separating if and only if it comes from a separating edge of $Z_\Lambda$.

\begin{proposition}\label{lem:FreeProductAmalgam} Let $\G=\G_1\sqcup\ldots\sqcup \G_k\sqcup \Lambda$ where  $\Lambda$ is discrete, and let $Y_\Gamma$ be a $\Gamma$-amalgam formed from $\G_i$-complexes  $X_{i}$ and a graph $Z_\Lambda$. Let $X_\Gamma$ be the complex formed from $Y_\Gamma$ by collapsing all separating edges from $Z_\Gamma$.  Then  $X_\G$ is a $\G$-complex, i.e. there exists a collection of $\Gamma$-partitions $\Pi$  such that  $X_\G\cong \Sa_\Gamma^\Pi$.
\end{proposition}

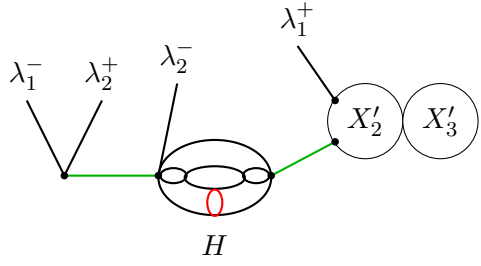
\begin{figure}\begin{center}
 \begin{tikzpicture} [scale=1]
 \coordinate (x) at (-.75,.275); \coordinate (y) at (0,.5); \coordinate (z) at (.75,.275);
 \coordinate (u) at (-2,.275); \coordinate (v) at (1.6,.725); \coordinate (w) at (1.6,1.275);
\begin{scope}[decoration={markings,mark = at position 0.5 with {\arrow{stealth}}}];
  \draw[thick, darkgreen] (z) to (v);
 \draw[thick, darkgreen] (x) to (u);
  \draw[thick] (1.1,2) to (w);\node[above] (l1) at (1.1,2) {$\lambda_1^+$};
 \draw[thick] (u) to (-2.5,1.275);\node[above] (L1p) at (-2.5,1.275) {$\lambda_1^-$};
\draw[thick] (u) to (-1.5,1.275);\node[above] (L2p) at (-1.5,1.275) {$\lambda_2^+$};
\draw[thick] (x) to (-.5,1.5);\node [above]  (L2m) at (-.5,1.5) {$\lambda_2^-$};
  \draw[fill=white] (2,1) circle(.5);
    \draw[fill=white] (3,1) circle(.5);
        \node (S2) at (2,1) {$X'_2$};
            \node (S3) at (3,1) {$X'_3$};
  \draw[thick] (0,.25) ellipse (.75cm and .5cm);
    \draw[thick] (0,.25) ellipse (.4cm and .15cm);
     \draw[thick] (-.55,.275) ellipse (.175cm and .1cm);
      \draw[thick] (.55,.275) ellipse (.175cm and .1cm);
       \vertex{(x)};
       \vertex{(z)};\vertex{(u)};\vertex{(v)};\vertex{(w)};
     \draw[thick,red] (0,-.085) ellipse (.1cm and .175cm);
     \node (H) at (0,-.65) {$H$};
         \end{scope}
 \end{tikzpicture}
\end{center}
\caption{ Determining the $\G$-partition associated to a hyperplane $H$ in $X_\G$. All hyperplanes in the treelike set for $X_1'$ other than $H$ have been collapsed in this figure  (see proof of Proposition~\ref{lem:FreeProductAmalgam})
}\label{fig:djunion}
 \end{figure}

\begin{proof}  First note that the hyperplanes of $X_\G$ consist of the hyperplanes of each $X'_{i}$ (which remain in $X'_{i}$) and the midpoints of non-separating edges of $Z_\Lambda$.   Choose a subdivided blowup structure on each $X'_{i}\subset Y_\G$, corresponding to a   compatible  multi-set of partitions $\Pi_i$.  Then choose a maximal tree $T_\Lambda$ in $Z_\Lambda$ and label and orient the edges not in $T_\Lambda$ by the elements of $\Lambda$.   Let $\mathcal T$ be the union of all hyperplanes with labels in the $\Pi_i$ and those dual to edges in $T_\Lambda$.   Collapsing all hyperplanes in $\mathcal T$ gives the Salvetti $\SaG$.

To see that $\mathcal T$ is the treelike set for a blowup structure, we need to check that each $H\in\mathcal T$ determines a \GW partition. Cut all edges of $Z_\Lambda$ that are labeled by elements of $\Lambda$,  labeling the initial half-edge $\lambda\inv$ and the terminal half-edge $\lambda$ (see Figure~\ref{fig:djunion}).  Each hyperplane in $\mathcal{T}$ 
now determines an evident partition of the vertices of $\G^{\pm}$.  First consider the partition associated to an edge $e$ of $T_\Lambda$.   Since we chose $T_\Lambda$ after collapsing each $X_i'$ to a point, no element of $\G_i$ is split by this partition. The fact that $Z_\Lambda$ has no separating edges implies that the partition associated to $e$ separates some $\lambda\in\Lambda$ from its inverse, so this gives a \GW partition based at $\lambda$.

A hyperplane in $X_i$ partitions the vertices  of $\G_i^\pm$, and the only new vertices of $\Gamma$ that might be split are in $\Lambda$, so have empty links, hence the resulting partition is still a valid \GW partition with its original base.

 We must also check that the duplicate partitions in the $\Pi_i$ give distinct $\G$-partitions, and the singleton partitions in the $\Pi_{i}$ give rise to legitimate  \GW\-partitions.  Singleton and duplicate partitions in $\Pi_{i}$ result from attaching an edge $e$ of $Z_\Lambda$ to the middle of a cube, which we then subdivide by duplicating all hyperplanes that intersect the cube. The edge $e$ lies on a characteristic cycle for some $w\in\Lambda$.

Let $H$ be one of the hyperplanes that has been  duplicated to form new hyperplanes $H'$ and $H''$, so that now $e$ terminates at a point between $H'$ and $H''$.    If $H$ was labeled by a partition $\WP$ in $\Pi_i$,  the new partitions $\WP'$ and $\WP''$ corresponding to $H'$ and $H''$ agree on  $\Delta^\pm$, but  either $w$ or $w\inv$ (depending on the orientation of the cycle) lies in opposite sides of the extensions of $\WP'$ and $\WP''$ to  $\G$-partitions.
  If $H$ was labeled by a vertex $v$, let $H'$ be the duplicate hyperplane that did {\em not} get the label  $v$, so $H'$ corresponds to a singleton partition $\mathcal S$.  Then $v$ (or $v\inv$) and $w$ (or $w\inv$) lie on the same side of the extension of $\mathcal S$ to $\G^\pm$, so the corresponding $\G$-partition is not a singleton.
 \end{proof}

In passing from $Y_\Gamma$ to $X_\Gamma$  in the proof of Proposition \ref{lem:FreeProductAmalgam}, we had to collapse the set of separating edges in $Z_\Gamma$. Since the collection of all separating hyperplanes in an NPC cube complex $X$ is acyclic and invariant under $\Aut(X)$,  Lemma~\ref{lem:CollapseSep}   implies that collapsing them has no effect on realizing actions of finite subgroups of $\Out(\AG)$ by isometries.

\begin{proposition}\label{lem:RealizingAmalgam}Let $\G=\G_1\sqcup\cdots \sqcup \G_k\sqcup\Lambda$ where $\Lambda$ is discrete and let $\rho\colon G\to U^0(\AG)$ be a homomorphism.
If  $(X_{i}, h_i)$  are marked  $\G_i$-complexes realizing the restriction of $\rho$ to $\G_i$, then there exists a marked $\G$-complex $(X_\G,h)$  realizing $\rho$ such that each $X_i$ is a subcomplex of $X_\G$ and   $h|_{X_i}=h_i$.
\end{proposition}

\begin{proof} Let $\overline{A}_\G$ and $\overline{A}_{\G_i}$ be the finite extensions of $A_\G$ and $A_{\G_i}$, respectively, determined by $G$. We apply Proposition 7.5 of \cite{HKFreeProd}
to $\G= \G_1\sqcup \cdots \sqcup \G_k\sqcup \Lambda$, with marked, complexes $(X_{i},h_i)$ as input. The result is a marked cube complex $X$ realizing the action of $G$ on $A_\Gamma$.
 In the construction (see Proposition 3.1 and Theorem 4.1 of \cite{HKFreeProd}), the marked complex $(X_\G, h)$ is formed from a graph of groups decomposition of $\overline{A}_\G$.  The edge stabilizers are all finite, vertex stabilizers are the corresponding finite extensions of the $\overline{A}_{\G_i}$, and the rank of the underlying graph is $|\Lambda|$.  Each vertex is then ``blown up" to a (possibly subdivided) copy of $X_{i}$ equipped with the marking $h_i$, to which edges are attached (see Remark 7.7 of \cite{HKFreeProd}), though we only subdivide where necessary, \emph{i.e.} where an added edge meets the interior of a cube.  In particular, collapsing each of the $X_{i}$ separately to points yields a graph of rank $|\Lambda|$. Thus, the hypotheses of Proposition \ref{lem:FreeProductAmalgam} are satisfied, and by Lemma \ref{lem:CollapseSep}, the resulting marked $\G$-amalgam $(X_\G, h)$ is a marked $\G$-complex which realizes $\rho$.
 \end{proof}

\begin{remark}\label{rem:subdivision}   In Proposition~\ref{prop:ExtendableAmalgam} the hypothesis that action of $G$ on $X_\Delta$  is reduced is essential.  In Proposition \ref{lem:RealizingAmalgam}, given a  graph $\G=\G_1\sqcup\ldots\sqcup \G_k\sqcup \Lambda$ and a collection of $\G_i$-complexes $X_i$, we constructed a $\G$-amalgam, $X_\G$, whose restriction to each $\G_i$ is a  subdivision of the given complexes $X_i$.  While we may assume that the original complexes $X_i$ are reduced, the resulting subdivisions need not be, since a subdivided  $e_v$ edge has one segment labeled $e_v$ and the rest labelled by partitions.  If $X_\G$ is not reduced, we can collapse an acyclic set of hyperplanes in $X_\G$ so that the resulting complex $X'_\G$ is reduced, and hence its restriction $X'_i$ to each ${\G_i}$-subcomplex is also reduced.  We claim that $X'_i$ is still a subdivision of the original $X_i$.  To see this, note that since the orbit of every edge $e_{\WP}$ in $X_i$ contains an $e_v$ edge, after subdividing, the orbit of at least one segment of this edge will also contains an $e_v$ edge.  Thus, this segment will not be collapsed in the reduction process.
\end{remark}

\section{Realizing finite subgroups of   $U^0(\AG)$}\label{sec:main}

 In this section we prove our main theorem.
 
 \begin{theorem}\label{thm:NRealization}  Let  $\G$ be a simplicial graph, $G$ a  finite group $G$  and $\rho\colon G\to U^0(\AG)$ a homomorphism.  Then there is a  $\Gamma$-complex $X$ with an untwisted marking $h\colon X\to\Sa_\G$ on which  $\rho$ is realized by isometries.
\end{theorem}

 We are especially interested in the case that $\rho$ is an inclusion, but   the proof is inductive and the general case is used in the induction.
The proof borrows a number of ideas from \cite{HK}.

\begin{proof}[Proof of theorem] For any $U^0$-invariant subgraph $\Delta$ we will say a marked $\Delta$-complex $(X_\Delta,h_\Delta)$ ``realizes $G$"  as a shorthand for ``realizes the  restriction $\rho_\Delta=r_\Delta\circ\rho\colon G\to U^0(A_\Delta).$"   Here the target of $h_\Delta$ is the subcomplex $\Sa_\Delta\subset\Sa_\G$.

We proceed by constructing  marked $\Delta$-complexes realizing $G$ for larger and larger $U^0$-invariant subgraphs $\Delta$ of $\G$, until we have one for all of $\G$.  Specifically, we define the {\em chain length}  $\ell=\ell(\Delta)$ to be the length  of  a maximal chain of $U^0$-invariant subgraphs $\emptyset=\G_0\subsetneqq \G_1\subsetneqq\cdots \subsetneqq \G_{\ell}\subsetneqq \Delta $   and proceed
 by induction on  $\ell(\Delta)$. At each stage we ensure that the $G$-action is reduced, so that the complexes we construct are extendable.

If $\ell(\Delta)=0$,  Proposition \ref{lem:MinimalInvariant}  says that  $\Delta$ is discrete, \emph{i.e.}  $A_{\Delta}$ is a free group.  The classical   Realization Theorem for free groups \cite{Cul,Khr,Zim} says we can find a marked graph $(X_\Delta,h_\Delta)$ on which $G$ is realized by isometries.
After collapsing $G$-invariant forests  we may assume  $X_\Delta$ is reduced (in particular has no separating edges), so  $(X_\Delta,h_\Delta)$ is the desired marked $\Delta$-complex.

Now  suppose $\ell(\Delta)=i\geq 1$.  Note that any proper  $U^0$-invariant subgraph of $\Delta$ has chain length strictly smaller than $i$, so we can construct marked complexes realizing $G$ for any such subgraph by induction.
Fix a maximal $U^0$-invariant chain $\emptyset=\G_0\subsetneqq \G_1\subsetneqq\cdots \subsetneqq \G_{i}\subsetneqq \Delta $ and let $\Theta=\Delta\setminus \G_{i}$. The next claim follows from maximality of the chain.

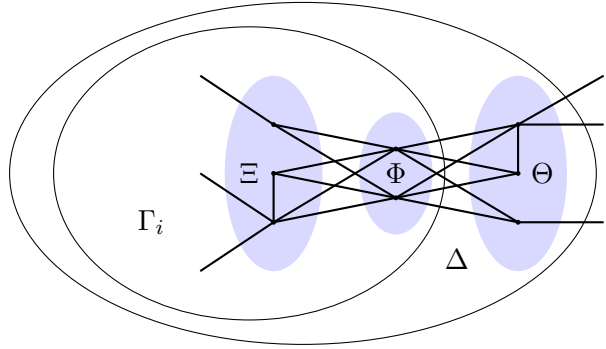
\begin{figure}
\begin{center}
\begin{tikzpicture}[scale=.65]
\draw (-3,0) ellipse (4cm and 3cm);
\draw (-1.9,0) ellipse (6cm and 3.5cm);
\fill[blue!15](0,0) ellipse (.75cm and 1.25cm);
\fill[blue!15] (-2.5,0) ellipse (1cm and 2cm);
\fill[blue!15] (2.5,0) ellipse (1cm and 2cm);
\coordinate(e1) at (-2.5,-1);\coordinate(e2) at (-2.5,0);\coordinate(e3) at (-2.5,1);
\coordinate(d1) at (0,-.5);\coordinate(d2) at (0,.5);
\coordinate(t1) at (2.5,-1);\coordinate(t2) at (2.5,0);\coordinate(t3) at (2.5,1);
\vertex{(d1)};\vertex{(d2)};
\vertex{(e1)};\vertex{(e2)};\vertex{(e3)};
\vertex{(t1)};\vertex{(t2)};\vertex{(t3)};
\draw[thick](e1) to (e2);
\draw[thick](t2) to (t3);
\draw[thick](t1) to (4.25,-1);
\draw[thick](t3) to (4.25,2);
\draw[thick](t3) to (4.25,1);
\draw[thick](e1) to (d1); \draw[thick](e1) to (d2);
\draw[thick](e2) to (d1); \draw[thick](e2) to (d2);
\draw[thick](e3) to (d1); \draw[thick](e3) to (d2);
\draw[thick](t1) to (d1); \draw[thick](t2) to (d1); \draw[thick](t3) to (d1);
\draw[thick](t1) to (d2); \draw[thick](t2) to (d2); \draw[thick](t3) to (d2);
\draw[thick] (e1) to (-4,-2);\draw[thick] (e1) to (-4,0); \draw[thick] (e3) to (-4,2);
\node (Xi) at (-3,0) {$\Xi$};
\node (Phi) at (0,0) {$\Phi$};
\node (Theta) at ( 3,0) {$\Theta$};
\node (Delta) at (1.25,-1.75) {$\Delta$};
\node (Gprime) at (-5,-1) {$\G_{i}$};
\end{tikzpicture}
\end{center}
\caption{Subgraphs of $\Delta$ referred to in the proof of Theorem~\ref{thm:NRealization}}\label{fig:LastStep2}
\end{figure}

{\bf Claim.}  If $w,w'\in\Theta$ then $\lk(w)\cap \G_{i}=\lk(w')\cap\G_{i}$.   
\begin{proof}[Proof of claim] By Proposition~\ref{cor:applications}(5), the subgraph spanned by $\G_{i}$ and all vertices adjacent to $\G_{i}$ is invariant, so by maximality of the chain either every element of $\Theta$ is adjacent to $\G_{i}$, or none of them are.  In the latter case, the claim is vacuously true because $\lk(w)\cap \G_{i}=\emptyset$ for every $w\in \Theta$.  In the former case, choose $w\in\Theta$ such that $\lk(w)\cap\G_{i}$ is maximal,  let $W=\{w'\in\Theta|\lk(w')\cap\G_{i}=\lk(w)\cap \G_{i}\}$ and let $\Delta'$ be the subgraph spanned by $\G_{i}$ and $W$.

  We will now show that $\Delta'$ is $U^0$-invariant, so by maximality of the chain $\Delta'=\Delta$ and $\Theta$ is the graph spanned by $W.$
 By Lemma~\ref{lem:InvariantSubgraph}, we need to prove
\begin{enumerate}
\item If $x \in \Delta'$ and $\lk(x) \subset \lk(y)$, then $y \in \Delta'$ and
\item If $\Delta'$ intersects more than one component of $\Gamma \backslash \st(y)$, then $y \in \Delta'$.
\end{enumerate}
We know that $\Delta$ is $U^0$-invariant, so under the hypotheses of either (1) or (2), $y$ must be in $\Delta$.  If $y \in \Gamma_i$, we are done, so assume instead that $y \in \Theta$. We need to show that $y \in W$.  If $x \in \Delta'$ and $\lk(x) \subset lk(y)$, the invariance of $\Gamma_i$,  {together with the assumption that $y \notin \Gamma_i$, guarantees that $x \notin \Gamma_i$.  That is,} $x \in W,$ so by maximality, $y$ is also in $W$.  In case (2), the invariance  of $\Gamma_i$ guarantees that $\st(y)$ does not separate $\Gamma_i$ so either it separates $\Gamma_i$ from $W$, or it separates two elements of $W$ from each other.  In either case, we must have $\lk(w) \subseteq \st(y)$, so
$\lk(w) \cap \Gamma_i \subseteq \lk(y) \cap \Gamma_i$, hence $y \in W$.
\end{proof}

Now set  $\Phi=\lk(\Theta)\cap \G_{i}$, and let $(X_{i},h_{i})$ be a reduced marked $\G_{i}$-complex  realizing $G$.    Note that $\Phi$ is $U^0$-invariant by  Proposition~\ref{cor:applications}.

 If $\Phi=\G_i$ then the above claim implies that $\Delta$ is the join  $\G_i\ast\Theta$, and that $\Theta=\lk(\G_i)\cap\Delta,$ so is $U^0$-invariant. By induction we can find  a reduced marked $\Theta$-complex $(X_\Theta,h_\Theta)$  which realizes the action of $G$, so by Proposition~\ref{lem:RealizingJoin} the product $(X_i\times X_\Theta, h_i\times h_\Theta)$ is a $\Delta$-complex realizing $G$. 
 
 If $\Phi=\emptyset$ then $\Delta$ is the disjoint union of $\G_i$ and $\Theta$, so by Proposition~\ref{lem:RealizingAmalgam} we can build a  $\Delta$-complex realizing $G$ using $(X_i,h_i)$ and  complexes for the  components of $\Theta$ that are not singletons (these components are $U^0$-invariant by Proposition~\ref{cor:applications}(4)). 

If $\Phi$ is a proper subgraph of $\G_i$, let $\Xi=\lk(\Phi)\cap \G_{i}$ (see  Figure~\ref{fig:LastStep2}).  Now   
$$\Delta= \G_i\cup \st_\Delta(\Phi)=\G_i\cup  \big(\Phi\ast (\Xi \sqcup\Theta)\big)$$ 
and $\G_i\cap\st_\Delta(\Phi)=\Phi\ast \Xi$ .  
Both $\Phi$  
and $\Xi$  
are $U^0$-invariant, so  $\Phi\ast\Xi$ is also $U^0$-invariant by Proposition~\ref{cor:applications}(3).  
Since $\lk_{\G_i}(\Phi\ast \Xi)=\emptyset$, $X_i$ contains a unique invariant (possibly subdivided) $(\Phi\ast\Xi)$-complex $X_{\Phi\ast\Xi}=X_\Phi\times X_\Xi$ realizing $G$,   by Corollary \ref{cor:NoLinkRestriction}.   {We may choose $h_i$ so that it restricts to a marking $h_{\Phi\ast \Xi}=h_{\Phi}\times h_{\Xi}$ on $X_{\Phi*\Xi}$.}

The subgraph $\st_\Delta(\Phi)$ is also $U^0$-invariant, and we   build a $\st_\Delta(\Phi)$-complex $X_{\st(\Phi)}$ realizing $G$ as follows.  We first build a complex $X_{\lk(\Phi)}$ for $\lk_\Delta(\Phi)=\Xi\sqcup\Theta$  using a copy of the complex $X_\Xi$ we already found in $X_i$ and  complexes for the  components of $\Theta$ that are not singletons, as we did in the case that $\Phi=\emptyset$ above.   {After reducing, Proposition \ref{prop:ExtendableAmalgam} ensures that $X_{\lk(\Phi)}$ is extendable. By Remark \ref{rem:subdivision}, the reduced complex still contains a subdivided copy of $X_{\Xi}$.} We then take  the  product of $X_{\lk(\Phi)}$ with a copy of the complex $X_\Phi\subset X_i$ to obtain a complex for $X_{\st(\Phi)}$ realizing $G$  {with respect to the product marking.}

If $\Xi\neq\emptyset$ then  $\lk_{\st(\Phi)}(\Phi\ast\Xi)=\emptyset,$ so the complex $X_{\st(\Phi)}$ that we just built contains a unique (possibly subdivided) $(\Phi\ast\Xi)$-complex realizing $G$.  By construction, this is identical to the complex $X_{\Phi\ast\Xi}$ contained in $X_i$, so we may glue $X_i$ to $X_{\st(\Phi)}$ by identifying these subcomplexes, thus forming a new complex $X_\Delta$.  {The markings $h_i$ and $h_{\st(\Phi)}$ agree on $X_{\Phi\ast\Xi}$, so we obtain a marked complex $(X_\Delta,h_\Delta)$ realizing $G$.}

If $\Xi=\emptyset$ then $\lk_{\st(\Phi)}(\Phi)=\Theta$, so $X_{\st(\Phi)}=X_\Phi\times X_\Theta$.  
 The following claim will allow us to pick out a particular slice $X_\Phi\times \{p\}\subset X_\Phi\times X_\Theta$ to glue to the (unique) copy of $X_\Phi$ contained in $X_i$.

{\bf Claim.} \label{lem:Claim9} The restriction  $\rho_\Theta\colon G\to U^0(A_\Theta)$  lifts to a homomorphism $G\to \Aut(A_\Theta)$.

\begin{proof}[Proof of claim] We use Lemma 2.2 of \cite{CCV}, which says 
\begin{itemize}
\item[($\ast$)]If $\Sigma$ is a subgraph of $\Delta$ then the normalizer in $A_\Delta$ of $A_\Sigma$ is $A_{\st_\Delta(\Sigma)}=A_\Sigma\times A_{\lk_\Delta(\Sigma)}$. 
\item[($\ast\ast$)] If $\Sigma_1,\Sigma_2\subset \Delta$  then $xA_{\Sigma_2} x\inv\leq  A_{\Sigma_1}$ if and only if  $\Sigma_2\subset \Sigma_1$ and $x=x_1x_2$ with $x_1\in N_{A_\Delta}(A_{\Sigma_1})$ and $x_2\in N_{A_\Delta}(A_{\Sigma_2})$. 
\end{itemize} Since all normalizers  will be taken with respect to $A_\Delta$, for the rest of this proof we omit $A_\Delta$ from the notation for normalizers.

We are assuming $\Xi=\emptyset,$ so $\Phi=\lk_\Delta(\Theta)$ and $\Theta=\lk_\Delta(\Phi)$.  Then ($\ast$) says $N(A_\Phi)=N(A_\Theta)=N(A_{\Phi*\Theta})=A_{\Phi*\Theta}$.  Furthermore,  if $\lk_\Delta(\G_i) \neq \emptyset$, then $\G_i = \Phi$ and $\Delta = \Phi *\Theta$.  We have already taken care of this case, so  we may assume that
$\lk_\Delta(\G_i)=\emptyset$ and we have  $N(A_{\G_i})=A_{\G_i}$.

Let $g\in G$.  Since $\G_i,$ $\Theta,$ $\Phi$ and $\Phi*\Theta$ are all $U^0$-invariant, the corresponding special subgroups of $A_\G$ are all sent to conjugates of themselves by any representative of $g$ in $\Aut(\AG)$.
Pick a representative $\widehat g$ that sends $A_{\G_i}$ to itself, and suppose $\widehat g(A_\Phi)=x  A_\Phi x\inv$.  Since $\widehat g(A_\Phi)\leq \widehat g(A_{\Gamma_i})=A_{\Gamma_i}$, ($\ast\ast$) says that $x=x_1x_2$ with $x_1\in N(A_{\G_i})=A_{\G_i}$ and $x_2\in N(A_\Phi)$.   Then
$x A_\Phi x\inv=x_1x_2  A_{\Phi}x_2\inv x_1\inv=x_1 A_{\Phi} x_1\inv$, so after composing $\widehat g$ with conjugation by $x_1$, we may assume $\widehat g$ sends both $A_{\G_i}$ and $A_\Phi$ to themselves.

Now $\Phi\subset\Phi*\Theta$ so $\widehat g(A_\Phi)\leq \widehat g(A_{\Phi*\Theta})=y A_{\Phi*\Theta}y\inv$ for some $y$. Since $\widehat g(A_\Phi)=A_\Phi$, this gives $y\inv A_\Phi y \leq A_{\Phi*\Theta}$, so by ($\ast\ast$) $y \in N(A_{\Phi*\Theta})N(A_\Phi)=A_{\Phi*\Theta}$, \emph{i.e.} $\widehat g(A_{\Phi*\Theta})=y  A_{\Phi*\Theta}y\inv=  A_{\Phi*\Theta}$.  Finally, $\Theta\subset\Phi*\Theta$ so $\widehat g(\Theta)=z A_\Theta z\inv \leq \widehat g(\Phi*\Theta)=\Phi*\Theta$ for some $z$, so ($\ast\ast$) says $z\in N(A_{\Phi*\Theta})N(A_\Theta)=A_{\Phi*\Theta}=N(A_\Theta)$, so $\widehat g(A_\Theta)=A_\Theta$ as well.

Now let $g_1,g_2\in G$ with $g_1g_2=g_3$,  and find representatives $\widehat g_1,\widehat g_2$ and $\widehat g_3$ as above.  We know that $\widehat g_1\widehat g_2\widehat g_3\inv$ is inner and preserves $A_{\G_i}$ and $A_\Theta$,  so the conjugating element lives in $N(A_{\G_i})\cap N(A_\Theta)=A_{\G_i}\cap A_{\Phi*\Theta}=A_\Phi$.  But conjugation by an element of $A_\Phi$ is trivial on $A_\Theta$, \emph{i.e.} the restriction of $\widehat g_1\widehat g_2\widehat g_3\inv$ to $A_\Theta$ is the identity. Thus $g\mapsto \widehat g$  gives a lift of $G$ to $\Aut(A_\Theta)$.
\end{proof}

By the claim,  the action of $G$ on $X_\Theta$ lifts to an action on $\widetilde{X}_\Theta$.  Since $\widetilde X_\Theta$ is $\textrm{CAT}(0)$ this action has a fixed point;  projecting this fixed point back down gives a fixed point $p\in X_\Theta$.
We now build $X_{\Delta}$ by gluing $X_{i}$ to $X_{\Phi*\Theta}=X_{\Phi}\times X_\Theta$ along their common subspace $X_\Phi = X_\Phi\times \{p\}$.  {As above, the markings $h_i$ and $h_{\Phi*\Theta}$ agree on the overlap, so give a marking $h_\Delta\colon X_\Delta\rightarrow \Sa_\Delta$ realizing $G$.}

 It remains to check that the complexes $X_\Delta$ that we have just built are actually $\Delta$-complexes.   We start by choosing a blowup structure  $X_i\iso \Sa_{\G_i}^{\Omega_i}$.  Since the action of $G$ on $X_i$ is reduced, this blowup structure is extendable, and induces (possibly subdivided) extendable blowup structures on the subcomplexes $X_\Phi, X_\Xi$ and $X_\Phi\times X_\Xi$.  
  
If  $\Xi=\emptyset$ then finding a blowup structure is slightly easier, so we do that case first. 
 In this case $\Theta$ is $U^0$-invariant, and we can choose an extendable blowup structure   $\Sa_\Theta^{\Omega_\Theta}$ on $X_\Theta$. Recall that we may have needed to subdivide $X_\Theta$ in order to make the fixed point a vertex; this means that $\Omega_\Theta$ may contain trivial or duplicate partitions. The fixed point $p\in X_\Theta$ now lies in a {\em region}, \emph{i.e.} a  consistent  choice of sides for each element of $\Omega_\Theta$.  The structure on $X_\Theta$  together with the blowup structure on $X_\Phi$ now give  a blowup structure on $X_\Phi\times X_\Theta$, by Proposition~\ref{lem:ProductBlowup}.
   The partitions in $\Omega_i$ and $\Omega_\Theta$ are all extendable, so in particular can be extended to $\Delta$.  We need to find extensions that form a compatible collection of $\Delta$-partitions.

If $\WP\in\Omega_i$ we need to decide where to put the vertices $v^\pm\in\Theta^\pm$ in our extension $\widehat\WP$. If $\WP$ is based at $m\in  \Phi$, they must all go into $\lk(\widehat\WP)$.   If  $m$ is distance at least 2 from $\Theta$,   there is some $u\in\Phi$ with $u\not\in \lk(m)$ (since $\Xi=\emptyset$), so all vertices of $\Theta$ are in the same component of $\Delta\setminus\st(m)$ as $u$, so all of $\Theta^\pm$ must go into the same side of $\WP$ as $u$.  (Note that the extendability of $\WP$ guarantees that since any two choices for $u$ lie in the same component of  $\Delta\setminus\st(m)$, they must lie on the same side of $\WP$, so there is no ambiguity here.)  We also need to extend the partitions  $\WQ\in \Omega_\Theta$ to $\Delta^\pm$.  All vertices in $\Phi$ must go into the link of each extension. Since $\G_i$ is $U^0$-invariant, no vertex in $\Theta$ has a star that separates $\G_i$, so vertices of $\G_i\setminus \Phi$ and their inverses all have to go in the same side of $\WQ$, for each $\WQ\in\Omega_\Theta$. We  put them all into the  region determined by the fixed point $p$.  It is now routine to check that all the extensions $\widehat\WQ$  and $\widehat\WP$ we have constructed are compatible.
This verifies that $X_\Delta$ is a $\Delta$-complex in the case $\Xi=\emptyset$.  If the $G$-action on $X_\Delta$ is not reduced, we can reduce it to obtain an extendable $\Delta$-complex.

If $\Xi\neq\emptyset$ we can find a blowup structure  $X_{\lk(\Phi)}\iso \Sa_{\lk(\Phi)}^{\Omega_{\lk(\Phi)}}$ that restricts to the given blowup structure on   $X_\Xi$ by By Proposition~\ref{lem:FreeProductAmalgam}.

The procedure we used in the case $\Xi=\emptyset$ to extend $\WP=(P_1|P_2|\lk(\WP))\in\Omega_i$ to $\Delta^\pm$ works again unless $\WP$ is based at $m\in\Xi$.   In this case  each $\WP\cap\Xi^\pm$ is also the restriction of a partition $\WQ=(Q_1|Q_2|\lk(\WQ))\in\Omega_{\st(\Phi)}$.  We form $\widehat\WP$ by adding $Q_i\cap\Theta^\pm$ to $P_i$ for $i=1,2$.

We also need to extend partitions $\WQ\in\Omega_{\st(\Phi)}$ that are based at $m\in\Theta$ to $\Delta$-partitions $\widehat\WQ$. The star $\st(m)$ cannot disconnect $\Xi$ or $\G_i$ since both are  $U^0$-invariant, so we add all of $\G_i\setminus \Phi$ to  the same side of $\widehat\WQ$ as $\Xi$.

 The extensions $\widehat\WP$ and $\widehat\WQ$ are now a compatible collection of $\Delta$-partitions, giving $X_{\Delta}$ a blowup structure.
 Reducing $X_\Delta$ if necessary, this completes the induction and  concludes the proof of the theorem.
\end{proof}

Recall from~\cite{CSV} that $K_\G$ is  a contractible subspace of $\OG$ which is invariant under the action of $U^0(\AG)$.  Points of $K_\G$ are $\G$-complexes with untwisted markings, so if $\rho$ is an inclusion, Theorem~\ref{thm:NRealization} gives the following statement.  

\begin{corollary}\label{cor:FixedPoint} The action of any finite subgroup of $U^0(A_\G)$ on the   $K_\G$ has a fixed point.  
\end{corollary}

Finally, we obtain the following information about finite subgroups of $U^0(\AG)$.

\begin{corollary}\label{cor:FinitelyMany} The group $U^0(A_\G)$ contains only finitely many conjugacy classes of finite subgroups.
\end{corollary}

\begin{proof} By Theorem~\ref{thm:NRealization} every finite subgroup of $U^0(\AG)$ is realized on a $\G$-complex.  Changing the marking produces a conjugate subgroup, so we may ignore the markings.  There are only finitely many $\G$-complexes since there are only finitely many partitions of $\G^\pm$, so only finitely many compatible collections of $\G$-partitions, each of which determines a unique blowup of $\SaG$.  Finally, there are only finitely many cubical isometries of a finite cube complex.
\end{proof}

\bibliography{OSbib.bib}
\bibliographystyle{alpha}

\end{document}